\documentclass[onecolumn,12pt,draftcls]{IEEEtran}
\IEEEoverridecommandlockouts

\usepackage{amsmath}
\usepackage{amssymb}
\usepackage{amsthm}
\usepackage{dsfont}
\usepackage{enumitem}
\usepackage{bm}
\usepackage{hyperref}

\usepackage{graphicx}
\usepackage[dvipsnames]{xcolor}

\newcommand{\ind}[1]{\mathds{1}\{#1\}}
\newcommand{\E}[2]{\mathbb{E} _{ #1 }  \left[ #2 \right]}
\newcommand{\Var}[2]{\mathrm{Var} _{#1} \left( #2 \right)}

\newcommand{\Prob}[2]{\mathbb{P} _{ #1 } \left\{ #2 \right\}}
\newcommand{\Ep}[1]{\mathrm{E}_{p}  \left[ #1 \right]}

\newcommand{\floor}[1]{\left\lfloor #1 \right\rfloor}
\newcommand{\norm}[1]{\left\lVert #1 \right\rVert}
\newcommand{\abs}[1]{\left| #1 \right|}
\newcommand{\R}{\mathbb{R}}

\newcommand{\td}{\Tilde}
\newcommand{\KL}[2]{D(#1||#2)}

\newcommand{\brc}[1]{\left( #1 \right)}
\newcommand{\sbrc}[1]{\left[ #1 \right]}
\newcommand{\cbrc}[1]{\left\{ #1 \right\}}
\newcommand{\eps}{\epsilon}
\newcommand{\calN}{\mathcal{N}}

\newcommand{\MISE}{\mathrm{MISE}}
\newcommand{\KLloss}{\text{KL-loss}}

\newcommand{\FAR}[1]{\mathrm{FAR} \left( #1 \right)}
\newcommand{\WADD}[1]{\mathrm{WADD} \left( #1 \right)}
\newcommand{\supp}{\mathrm{supp}}

\newcommand{\WADDth}[1]{\mathrm{WADD} ^{ \theta } \left( #1 \right)}

\DeclareMathOperator*{\esssup}{ess\,sup}

\newtheorem{theorem}{Theorem}[section]
\newtheorem{corollary}{Corollary}[theorem]
\newtheorem{lemma}[theorem]{Lemma}

\theoremstyle{definition}

\newtheorem{example}{Example}[section]

\theoremstyle{remark}
\newtheorem*{remark}{Remark}

\begin{document}

\title{Quickest Change Detection with Post-Change Density Estimation}

\author{Yuchen Liang, ~\IEEEmembership{Member,~IEEE}\thanks{Y. Liang was with the ECE Department and Coordinated Science Laboratory, University of Illinois at Urbana-Champaign; he is now with 
the ECE Department, The Ohio State University; email: liang.1439@osu.edu.}
%George V. Moustakides, ~\IEEEmembership{Life Senior Member,~IEEE},
%\thanks{George V. Moustakides is with the Electrical and Computer Engineering Department, University of Patras, 26500 Patras, Greece; e-mail: moustaki@upatras.gr} 
and  Venugopal V. Veeravalli, ~\IEEEmembership{Fellow, IEEE}\thanks{V.~V.~Veeravalli is with the ECE Department and Coordinated Science Laboratory, 
University of Illinois at Urbana-Champaign; email: vvv@illinois.edu.}

\thanks{This work was supported in part by the US National Science Foundation under grant ECCS-2033900.}
%, and by the Army Research Laboratory under Cooperative Agreement W911NF-17-2-0196, through the University of Illinois at Urbana-Champaign.}

\thanks{
This work is based in part on the paper presented at the 2023 IEEE International Conference on Acoustics, Speech and Signal Processing~(ICASSP)~\cite{liang2023icassp} }
}

\maketitle

\vspace*{-0.5in}

\begin{abstract}
The problem of quickest change detection in a sequence of independent observations is considered. The pre-change distribution is assumed to be known, while the post-change distribution is unknown. Two tests based on post-change density estimation are developed for this problem, the window-limited non-parametric generalized likelihood ratio (NGLR) CuSum test and the non-parametric window-limited adaptive (NWLA) CuSum test. Both tests do not assume any knowledge of the post-change distribution, except that the post-change density satisfies certain smoothness conditions that allows for efficient non-parametric estimation.
Also, they do not require any pre-collected post-change training samples. Under certain convergence conditions on the density estimator, it is shown that both tests are first-order asymptotically optimal, as the false alarm rate goes to zero.  
The analysis is validated through numerical results, where both tests are compared with baseline tests that have distributional knowledge.
\end{abstract}


\section{Introduction}

The problem of quickest change detection (QCD) is of fundamental importance in mathematical statistics (see, e.g., \cite{vvv_qcd_overview,xie_vvv_qcd_overview} for an overview). Given a sequence of observations whose distribution changes at some unknown change-point, the goal is to detect the change in distribution as quickly as possible after it occurs, while controlling the false alarm rate. 
In classical formulations of the QCD problem, it is assumed that the pre- and post-change distributions are known, and that the observations are independent and identically distributed (i.i.d.) in the pre- and post-change regimes. However, in many practical situations, while it is reasonable to assume that we can accurately estimate the pre-change distribution, the post-change distribution is rarely completely known.

There have been extensive efforts to address pre- and/or post-change distributional uncertainty in QCD problems. In the case where both distributions are not fully known, one approach is to assume that the distributions are parametrized by a (low-dimensional) parameter that comes from a pre-defined parameter set, and to employ a generalized likelihood ratio (GLR) approach for detection. This approach was first introduced in \cite{lorden1971} and later analyzed in more detail in \cite{lai1998infobd}.
In particular, in \cite{lai1998infobd}, it is assumed that the pre-change distribution is known and that the post-change distribution comes from a parametric family, with the parameter being finite-dimensional. A window-limited GLR test is proposed, which is shown to be asymptotically optimal under certain smoothness conditions.
This work has recently been extended to non-stationary post-change settings \cite{liang2022nonstat}. For the setting considered in \cite{lai1998infobd}, a window-limited adaptive approach to constructing a QCD test  was developed in recent work \cite{xie2022windowlimited}. This adaptive test is also shown to achieve first-order asymptotic optimality \cite{xie2022windowlimited}. In this paper, one of the test constructions for the case where the post-change is completely unknown is based on extending techniques introduced in \cite{xie2022windowlimited}.

%Another approach to dealing with distributional uncertainty in QCD problems is the minimax robust approach \cite{huber1965}, where it is assumed that the pre- and post-change distributions come from (known) mutually exclusive uncertainty classes, and the goal is to optimize the performance for the worst-case choice of distributions in the uncertainty classes. Under certain conditions, e.g., joint stochastic boundedness (see, e.g., \cite{moulin-veeravalli-2018} for a definition) and weak stochastic boundedness \cite{molloy2017asymrobustqcd}, robust solutions can be found \cite{unnikrishnan2011robustqcd, molloy2017asymrobustqcd}. However, these robust tests can have suboptimal performance for the actual distributions encountered in practice.

We assume complete knowledge of the pre-change distribution, while not making any parametric assumptions about the post-change distribution. 
There has been prior work along these lines. One approach is to replace the log-likelihood ratio by some other useful statistic for distinguishing between distributions in constructing tests.  Examples of this approach include the use of kernel M-statistics \cite{xie2015mstat, flynn2019kernelcusum}, one-class SVMs \cite{desobry2005onesvm}, nearest neighbors \cite{chu2022sequential, hao2019knn}, and Geometric Entropy Minimization \cite{yilmaz2017gem, kurt2020gem}. In \cite{xie2015mstat}, a test is proposed that compares the kernel maximum mean discrepancy (MMD) within a window to a given threshold. A way to set the threshold is also proposed that meets the false alarm rate asymptotically \cite{xie2015mstat}. Another approach is to estimate the log-likelihood ratio and thus the CuSum test statistic through a pre-collected training dataset. This include direct kernel estimation \cite{sugiyama2012direct} and, more recently, neural network estimation \cite{moustakides2019training}. 
%In \cite{sugiyama}, a test is proposed where its statistic comes from the direct kernel-based density-ratio estimation. 
However, \textit{the tests proposed in \cite{xie2015mstat}--\cite{moustakides2019training} lack explicit performance guarantees on the detection delay.}
%In \cite{lau2018binning}, a binning approach is proposed to solve the QCD problem asymptotically without any pre-collected training set. The asymptotic optimality is established for the case where the pre-change distribution is known, the post-change distribution is distinguishable from the pre-change with binning, and both distributions have discrete support.

Our contributions are as follows:
\begin{enumerate}
    \item We propose a window-limited non-parametric generalized likelihood ratio (NGLR) CuSum test and a non-parametric window-limited adaptive (NWLA) CuSum test, both of which do not assume any knowledge of the post-change distribution (except that the post-change density satisfies certain smoothness conditions that allows for efficient non-parametric estimation), and do not require any post-change training data.
    \item We characterize a generic class of density estimators that enable detection. 
    \item For both tests, we provide a way to set the test threshold to meet false alarm constraints (asymptotically).
    \item We show that both proposed tests are first-order asymptotically optimal with the selected thresholds, as the false alarm rate goes to zero.
 %   \item We characterize the convergence rate in the delay of the WLA CuSum test as a function of the false alarm rate.
    \item We validate our analysis through numerical results, in which we compare both tests with baseline tests that have distributional knowledge.
\end{enumerate}

The rest of the paper is structured as follows. In Section~\ref{sec:de_property}, we describe some properties required of the density estimators for asymptotically optimal QCD. In Section~\ref{sec:qcd_loo}, we propose the NGLR-CuSum test and analyze its theoretical performance. In Section~\ref{sec:qcd_adp}, we study the performance of the NWLA-CuSum test. Both tests are analyzed under the assumption that the post-change distribution is completely unknown.
In Section~\ref{sec:num_res}, we present numerical results that validate the theoretical analysis.
In Section~\ref{sec:concl}, we provide some concluding remarks. 

A preliminary version of the results in this paper for the NGLR-CuSum test appeared in \cite{liang2023icassp}.

\section{Density Estimators for Quickest Detection}
\label{sec:de_property}

Let $X_1,X_2,\dots \in \R^d$ be i.i.d. observations drawn from an unknown distribution, with probability density function (or \emph{density}) $p$ with respect to some dominating measure $\mu$, and let $\supp(p)$ be the support of $p$. Let $\mathrm{E}_{p}$ and $\mathrm{V}_{p}$ denote, respectively, the expectation and variance operator on the sequence of observations, when the density corresponding to each observation is $p$. For two densities $p$ and $q$ on $\R^d$ with respect to $\mu$, the Kullback-Leibler (KL) divergence is defined as:
\[ 
\KL{p}{q} := \int_{\supp(p)} \log (p(x)/q(x)) p(x) d \mu(x). 
\]
Define $X^{[k,n]} := X_k,\dots,X_n$. Let $\widehat{p}^{n,k}_{-i}$ be a density on $\R^d$ with respect to $\mu$ that is estimated using $X^{[k,n]}_{-i} := X_k,\dots, X_{i-1},X_{i+1},\dots,X_n$, where the subscript $-i$ represents that $X_i$, with $k \leq i \leq n$, is the observation that is left out from $X^{[k,n]}$.  We refer to $\widehat{p}^{n,k}_{-i}$ as a leave-one-out (LOO) estimator. Note that $\widehat{p}^{n,k}_{-i}$ and $X_i$ are independent for each $1 \leq k \leq i \leq n$. 
% Also, \vvv{I don't think the following equation is correct. The expectation is the integral of the product $\widehat{p}^{n,k}_{-i}(x) p(x)$, which is not equal to 1. But I don't think you need it to equal 1 anyway, and so you can skip this equation.}
% \[ 
% \Ep{\widehat{p}^{n,k}_{-i}(X_i)} = 
% \Ep{\Ep{\widehat{p}^{n,k}_{-i}(X_i) | X^{[k,n]}_{-i}}}= 
% 1. \]
% The estimation procedure is assumed to be sample-homogeneous, i.e., $\widehat{p}^{n,k}_{-i} \stackrel{d.}{=} \widehat{p}^{n,k}_{-j}, \forall k \leq i < j \leq n$. 

With some possible abuse of notation, we also define 
\[
\widehat{p}^{w}_n := \widehat{p}^{n,n-w}_{-n}
\]
to be the LOO estimate of $p$ obtained from the past $w$ i.i.d. samples from $p$.
% In this section, we write $\widehat{p}^{w} = \widehat{p}^{w}_n$ for brevity.

Suppose that, for large enough $w$, there exist constants $0 < \beta_1,C_1,C_2 < \infty$ and $0 < \beta_2 < 2$ (that depend only on the density $p$ and the estimation procedure) such that the KL loss \cite{hall1987klloss} of the density estimator satisfies
\begin{equation}
\label{eq:converg_m1_bound}
    \text{KL-loss} (\widehat{p}^{w}_n) := \Ep{\KL{p}{\widehat{p}^{w}_n}} \leq \frac{C_1}{w^{\beta_1}}
\end{equation}
where the KL divergence and the expectation operator $\mathrm{E}_p$ are taken over the randomness of $X_n$ and $\widehat{p}^{w}_n$, respectively.
% In the context of change detection, if $p = p_0$, equation~\eqref{eq:converg_m1_bound} is equivalent to
% \begin{equation}
%     \E{\infty}{\widehat{Z}_n} \leq - \frac{c_1}{w^{\overline{\beta}_1}}.
% \end{equation}
% When $p = p_1$, equation~\eqref{eq:converg_m1_upper_bound} is equivalent to
% \begin{equation}
%     \widehat{I} \geq I - \frac{C_1}{w^{\underline{\beta}_1}}.
% \end{equation}
% This guarantees that $\widehat{I} > 0$ for some sufficiently large $w$.
Also, the second moment satisfies
\begin{equation}
\label{eq:converg_m2_bound}
    \Ep{\brc{\log \frac{p(X_n)}{\widehat{p}^{w}_n(X_n)}}^2} \leq \frac{C_2}{w^{\beta_2}}.
\end{equation}
Here the expectation operator $\mathrm{E}_p$ is taken over the randomness of both $X_n$ and $\widehat{p}^{w}_n$. Recall that $X_n$ is independent of $\widehat{p}^{w}_n$.
% Finally, it is assumed that $\widehat{p}^{w}_n$ integrates to 1, i.e.,
% \[ \Ep{\widehat{p}^{w}_n(X)} = 1 \]
% where $X$ is similarly independent of $\widehat{p}^{w}_n$.

Similar assumptions to \eqref{eq:converg_m1_bound} and \eqref{eq:converg_m2_bound} are imposed for general $\widehat{p}^{n,k}_{-i}$ as follows. When $n-k$ is large enough, for each $k \leq i \leq n$, 
\begin{equation}
\label{eq:converg_m1_bound_loo}
    \KLloss(\widehat{p}^{n,k}_{-i}) = \Ep{\KL{p}{\widehat{p}^{n,k}_{-i}}} \leq \frac{C_1}{(n-k)^{\beta_1}}
\end{equation}
and
\begin{equation}
\label{eq:converg_m2_bound_loo}
    \Ep{\brc{\frac{1}{(n-k+1)}\sum_{i=k}^{n} \log \frac{p(X_i)}{\widehat{p}^{n,k}_{-i}(X_i)}}^2} \leq \frac{C_2}{(n-k+1)^{\beta_2}}.
\end{equation}

A typical loss measure for a density estimator is the mean-integrated squared error (MISE), defined as (see, e.g., \cite[Chap.~2]{scott2015mult-denst-est})
\begin{align}
\label{eq:mise_def}
   \MISE(p, \widehat{p}^{w}_n) &= \Ep {\int (\widehat{p}^{w}_n (x_n) - p(x_n))^2 d \mu(x_n)} 
   = \Ep {\norm{\widehat{p}^{w}_n - p}_2^2}.
\end{align}
The following lemma connects the MISE measure with the bounds in \eqref{eq:converg_m1_bound}--\eqref{eq:converg_m2_bound_loo}. The proof is given in the Appendix.
\begin{lemma}
\label{lem:Dest}
Suppose that there exist $\overline{\zeta}, \underline{\zeta}$ such that 
\begin{equation} \label{eq:compact_support}
    0 < \underline{\zeta} \leq p(x), \widehat{p}^{w}_n(x) \leq \overline{\zeta} < \infty,~\forall x \in \supp(p).
\end{equation}
If the estimator achieves
\begin{equation} \label{eq:mise_est_cond}
    \MISE(p, \widehat{p}^{w}_n) \leq \frac{C_3}{w^{\beta_3}}, 
\end{equation}
for all $w$ large enough and for some constants $0 < \beta_3, C_3 < \infty$, then
\eqref{eq:converg_m1_bound}-\eqref{eq:converg_m2_bound_loo} 
are satisfied with
\[ C_1 = \frac{C_3}{\underline{\zeta}},\quad C_2 = \frac{\overline{\zeta} r C_3}{\underline{\zeta}^2}, \quad \beta_1 = \beta_2 = \beta_3 \]
where 
\begin{equation}\label{eq:rdef}
r := \brc{\frac{\log (\underline{\zeta} / \overline{\zeta})}{(\underline{\zeta} / \overline{\zeta})-1}}^2.
\end{equation}
\end{lemma}

% \begin{remark}
% Under the assumption of \eqref{eq:compact_support}, the maximum likelihood estimator (MLE) in the \textit{parametric} density estimation problem (if it exists) achieves $\beta_1 = \beta_2 = \frac{1}{2}$ (e.g., see \cite{moulin-veeravalli-2018}).
% \end{remark}
In the following, for any positive functions $g(w),h(w)$, the notation $h(w) = O(g(w))$ means that $\frac{h(w)}{g(w)} \xrightarrow{w \to \infty} L < \infty$, and $h(w) = \Omega(g(w))$ means that $\frac{h(w)}{g(w)} \xrightarrow{w \to \infty} L > 0$.

\begin{corollary} \label{cor:lem_Dest}
Suppose that \eqref{eq:compact_support} is satisfied with 
% \vvv{I don't like the notation $\overline{\beta}$ and $\underline{\beta}$. Could you change these to $\underline{\beta}$ and $\overline{\beta}$?}
\[ \overline{\zeta} = \overline{\zeta}_w = O(w^{\overline{\beta}}),\quad \underline{\zeta} = \underline{\zeta}_w = \Omega(w^{-\underline{\beta}})\]
such that
\[ \underline{\beta} < \beta_3 / 2, \quad \overline{\beta} < \beta_3 - 2 \underline{\beta}. \]
 Suppose that the estimator still achieves \eqref{eq:mise_est_cond}. Then, \eqref{eq:converg_m1_bound}--\eqref{eq:converg_m2_bound_loo} are still satisfied, with
\[ \beta_1 = \beta_3 - \underline{\beta},\quad \beta_2 = \beta_3 - 2 \underline{\beta} - \overline{\beta} - \varrho \]
where $\varrho > 0$ is a small constant such that $\beta_2$ is still positive.
\end{corollary}
The proof of this corollary is given in the Appendix.

An example of a density estimator that satisfies \eqref{eq:converg_m1_bound}--\eqref{eq:converg_m2_bound_loo}  (under condition \eqref{eq:compact_support} and when the density satisfies some smoothness condition) is the \emph{kernel} density estimator (KDE).
\begin{example}[Kernel Density Estimator (KDE)] \label{ex:kde}
Suppose $d = 1$ and the dominating measure $\mu$ is the Lebesgue measure on $\R$. Given observations $X_1,\dots,X_w$, the kernel density estimator (KDE) is defined as
\begin{equation}
\label{def:kde}
    \widehat{p}^w_n(x_n) = \frac{1}{w h} \sum_{j=n-w}^{n-1} K\left(\frac{x_n-X_j}{h}\right) % =: \frac{1}{w} \sum_{j=1}^w K_h\left(x-X_j\right)
\end{equation}
where $K(\cdot) \geq 0$ is a kernel function and $h > 0$ is a smoothing parameter.
% It can be shown that KDE is rate optimal for the $\gamma$-H\"older class of densities. 
Define the $\gamma$-H\"older density class as
\[ {\cal H}_\gamma := \left\{ p: \int p(x) d x = 1, \exists L > 0, \abs{p^{(\ell)}(x_1) - p^{(\ell)}(x_2)} \leq L \abs{x_1 - x_2}^{\gamma - \ell},~\forall x_1,x_2 \in \supp(p) \right\}. \]
Here $\gamma > 0$ and $\ell = \floor{\gamma}$. Further, if the kernel function $K(\cdot)$ satisfies
\begin{equation} \label{ex:kde_kernel_order}
    \int K(u) d u = 1,\quad \int u^j K(u) d u = 0,~j = 1, \ldots, \ell.
\end{equation}
Then, as shown in \cite{tsybakovi2009intro_np_est}, with a properly chosen $h = h^w$, the KDE satisfies
\[ \sup_{p \in {\cal H}_\gamma} \MISE(p,\widehat{p}^w_n) = O(w^{-\frac{2 \gamma}{2 \gamma + 1}}). \]
% while, for any density estimator $T^w$ using $w$ samples,
% \[ \inf_{T^w} \sup_{p \in {\cal H}_\gamma} \MISE(p,T^w) = \Omega(w^{-\frac{2 \gamma}{2 \gamma + 1}}). \]
Therefore, if the condition \eqref{eq:compact_support} is further satisfied, from Lemma~\ref{lem:Dest}, conditions \eqref{eq:converg_m1_bound}--\eqref{eq:converg_m2_bound_loo} are satisfied with
\begin{equation}
\label{ex:kde_order1}
    \beta_1 = \beta_2 = \frac{2 \gamma}{2 \gamma + 1}.
\end{equation}

For the case where $\mu$ is the Lebesgue measure on $\R^d$, a product kernel can be used to estimate the density, and the corresponding KDE is
\begin{equation*}
    \widehat{p}^w_n(\bm{x}_n) = \frac{1}{w \prod_{i=1}^d h^{(i)}} \sum_{j=n-w}^{n-1} \prod_{i=1}^d K\left( \frac{x_n^{(i)}-X_j^{(i)}}{h^{(i)}}\right)
\end{equation*}
where $x^{(i)},~i=1,\dots,d$ is the $i$-th element of a vector $\bm{x} \in \R^d$, and $\bm{h}$ is a vector for smoothing parameter.
With a properly chosen $\bm{h}$, it can be shown that \cite{wasserman2006all-nonpara-stat}:
\[ \sup_{p \in {\cal H}_\gamma} \MISE(p,\widehat{p}^w_n) = O(w^{-\frac{2 \gamma}{2 \gamma + d}}). \]
Therefore, if the condition \eqref{eq:compact_support} is further satisfied, we have
\begin{equation}
\label{ex:kde_orderd}
    \beta_1 = \beta_2 = \frac{2 \gamma}{2 \gamma + d}.
\end{equation}
% \begin{equation}
%     \widehat{p}^{n,k}_{-i}(x_i) = \frac{1}{(n-k)h} \sum_{\substack{j=k \\ j \neq i}}^n K\left(\frac{x_i-x_j}{h}\right)
% \end{equation}
% \begin{equation*}
%     \widehat{p}^{n,k}_{-i}(\bm{x}_i) = \frac{1}{(n-k)\prod_{l=1}^d h_l} \sum_{\substack{j=k \\ j \neq i}}^n \prod_{l=1}^d K\left( \frac{\bm{x}_{i,l}-\bm{x}_{j,l}}{h_l}\right)
% \end{equation*}
% where $K(\cdot) \geq 0$ is a kernel function and $h > 0$ is a smoothing parameter.
\end{example}
% The KL loss for kernel density estimators is analyzed carefully in \cite{hall1987klloss}, where it is shown that the rate of convergence in KL loss is slower than that of MISE for most well-behaved densities. Nevertheless, the KL loss indeed converges to zero with a polynomial decay rate with the use of appropriate kernel functions, and thus \eqref{eq:converg_m1_bound} is satisfied.
% Furthermore, it can be shown that \eqref{eq:converg_m2_bound} is also satisfied under the assumption of \eqref{eq:compact_support}. 
We note that the actual choices of $\beta_1$ and $\beta_2$ do not affect the first-order asymptotic optimality results given in Thm~\ref{thm:opt_loo} and Thm~\ref{thm:opt_wla}. 
%We also stress that the our asymptotic results is not restrictive to any particular choice of density estimator.

\section{QCD with NGLR-CuSum Test}
\label{sec:qcd_loo}

Let $X_1,X_2,\dots,X_n,\dots \in \R^d$ be a sequence of independent random variables (or vectors), and let $\nu$ be a change-point. Assume that $X_1, \dots, X_{\nu-1}$ all have density $p_0$ with respect to some dominating measure $\mu$. Furthermore, assume that $X_\nu, X_{\nu+1}, \dots$ have densities $p_1$ also with respect to $\mu$. Here $p_0$ is assumed to be completely known. Regarding $p_1$, we only assume that \eqref{eq:converg_m1_bound_loo} and \eqref{eq:converg_m2_bound_loo} are satisfied. Let $({\cal F}_{n})_{n\ge 0}$ be the filtration, with ${\cal F}_{0}=\{\Omega,\varnothing\}$ and ${\cal F}_{n}=\sigma\left\{X_{\ell}, 1\le \ell \le n \right\}$ being the sigma-algebra generated by the set of $n$ observations $X_1,\dots,X_n$. Furthermore, let ${\cal F}_\infty= \sigma(X_1,X_2, \dots)$. 

Let $\mathbb{P}_\nu$ denote the probability measure on the entire sequence of observations when the change-point is $\nu$, and let $\mathbb{E}_{\nu}$ denote the corresponding expectation.
The change-time $\nu$ is assumed to be unknown but deterministic. The problem is to detect the change quickly, while controlling the false alarm rate. Let $\tau$ be a stopping time \cite{moulin-veeravalli-2018} defined on the observation sequence associated with the detection rule, i.e. $\tau$ is the time at which we stop taking observations and declare that the change has occurred.

We employ standard notations as follows:
\begin{align*}
    &o(x)~\text{as}~x \to x_0 ~\text{for a function $h(x) \geq 0$ such that}~ \limsup_{x \to x_0} \abs{\frac{h(x)}{x}} = 0\\
    &\omega(x)~\text{as}~x \to x_0 ~\text{for a function $h(x) \geq 0$ such that}~ \liminf_{x \to x_0} \abs{\frac{h(x)}{x}} = \infty\\
    &O(x)~\text{as}~x \to x_0 ~\text{for a function $h(x) \geq 0$ such that}~ \limsup_{x \to x_0} \abs{\frac{h(x)}{x}} < \infty\\
    % &\Omega(x)~\text{as}~x \to x_0 ~\text{for the function $h(x) \geq 0$ such that}~ \liminf_{x \to x_0} \abs{\frac{h(x)}{x}} > 0\\
    &\Theta(x)~\text{as}~x \to x_0 ~\text{for a function $h(x) \geq 0$ such that}~ \lim_{x \to x_0} \abs{\frac{h(x)}{x}} = L \in (0,\infty)
\end{align*}
and $A_\alpha\sim B_\alpha$ is equivalent to $A_\alpha = B_\alpha (1+o(1))$. If not explicitly specified, $x \to x_0$ refers to $\alpha \to 0$ or $b \to \infty$.

\subsection{QCD Problem Formulation and Classical Results}

When $p_1$ is known, Lorden \cite{lorden1971} proposed solving the following optimization problem to find the best stopping time $\tau$:
\begin{equation}
\label{prob_def}
    \inf_{\tau \in \mathcal{C}_\alpha} \WADD{\tau}
\end{equation}
where
\begin{equation} \label{eq:def_wadd}
    \WADD{\tau} := \sup_{\nu \geq 1} \esssup \E{\nu}{\left(\tau-\nu+1\right)^+|{\cal F}_{\nu-1}}
\end{equation}
characterizes the worst-case delay, and the constraint set is
\begin{equation}
\label{fa_constraint}
    \mathcal{C}_\alpha := \left\{ \tau: \FAR{\tau} \leq \alpha \right\}
\end{equation}
with $\FAR{\tau} := \frac{1}{ \E{\infty}{\tau}}$,
which guarantees that the false alarm rate of the algorithm does not exceed $\alpha$. Here, $\E{\infty}{\cdot}$ is the expectation operator when the change never happens, and $(\cdot)^+:=\max\{0,\cdot\}$.

Lorden also showed that Page's Cumulative Sum (CuSum) algorithm \cite{page1954} whose test statistic is given by:
\begin{equation*}
    W(n) = \max_{1\leq k \leq n} \sum_{i=k}^n \log \frac{p_1(X_i)}{p_0(X_i)} = \brc{W(n-1)}^+ + \log \frac{p_1(X_n)}{p_0(X_n)}
\end{equation*}
solves the problem in \eqref{prob_def} asymptotically as $\alpha \to 0$.
The CuSum stopping rule is given by:
\begin{equation}
\label{def:cusum}
    \tau_{\text{Page}}\left(b\right) := \inf \{n:W(n)\geq b \}.
\end{equation}
It was shown by Moustakides \cite{moustakides1986optimal} that the CuSum test is exactly optimal for the problem in \eqref{prob_def} with some threshold $b_\alpha$ that meets the false alarm constraint exactly, where $b_\alpha \sim \abs{\log \alpha}$. Thus, we have the first-order asymptotic approximation as:
\begin{equation}
    \inf_{\tau \in \mathcal{C}_\alpha} \WADD{\tau} = \WADD{\tau_{\text{Page}}\left(b_\alpha\right)} \sim \frac{\abs{\log \alpha}}{I}
\end{equation}
as $\alpha \to 0$. Here we define
\[ I := \KL{p_1}{p_0}. \]

When the post-change distribution has parametric uncertainties, Lai \cite{lai1998infobd} generalized this performance guarantee with the following assumptions. 
% \vvv{The rest of this subsection needs to be rewritten since the $Z_i$ below should be a function of $\theta$. Also, the definition of WADD needs to be specified as a function of $\theta$, and the asymptotic optimality is for all $\theta \in \Theta$, under some smoothness conditions. One option is to skip this discussion altogether and give the conditions required on the true LLR in the following subsections, without talking about the parametric case at all.} \yuchen{In the experiment we indeed used the GLR test and compared our NGLR against it. So I rewrote the rest of this subsection to incorporate parametric uncertainties. I also realized that I did not formally define $Z_i$, so I removed those.}  \vvv{Okay, but there were still some issues with the problem formulation below, which I have fixed. }
Let $\theta \in \Theta$ be the post-change parameter, and denote the post-change density as $p_1^\theta$. Define $\mathbb{P}^\theta_\nu$ and $\mathbb{E}^\theta_\nu$ to be the probability and expectation operator on the sequence, respectively, when the true post-change density is $p_1^\theta$. For fixed $\theta \in \Theta$, define the worst-case average detection delay as:
\begin{equation}\label{eq:WADD_th_def}
    \WADDth{\tau} := \sup_{\nu \geq 1} \esssup \mathbb{E}_{\nu}^\theta \sbrc{\left(\tau-\nu+1\right)^+|{\cal F}_{\nu-1}}.
\end{equation}
Under parametric uncertainty, the goal is to find a test that belongs to $\mathcal{C}_\alpha$ (see \eqref{fa_constraint} and achieves  
\begin{equation} \label{eq:opt_prob_alph}
    \inf_{\tau \in \mathcal{C}_\alpha} \WADDth{\tau}
\end{equation}
for every $\theta \in \Theta$.

Define $I^\theta := \KL{p_1^\theta}{p_0}$. Suppose that $p_0$ and $p_1^\theta$ satisfy
\begin{equation}
\label{eq:lai_upper}
    \sup_{\nu \geq 1} \mathbb{P}_{\nu}^\theta \cbrc{\max_{t \leq n} \sum_{i=\nu}^{\nu+t} \log \frac{p_1^\theta(X_i)}{p_0(X_i)} \geq (1+\delta) n I^\theta} \xrightarrow{n \to \infty} 0
\end{equation}
for any $\delta > 0$, and
\begin{equation}
\label{eq:lai_lower}
    \sup_{t \geq \nu} \mathbb{P}_{\nu}^\theta \cbrc{\sum_{i=t}^{t+n} \log \frac{p_1^\theta(X_i)}{p_0(X_i)} \leq (1-\delta) n I^\theta} \xrightarrow{n \to \infty} 0
\end{equation}
for any $\delta \in (0,1)$. Also, suppose that the window size $m_\alpha$ satisfies
\begin{equation*}
\label{eq:lai_malpha}
    \liminf m_\alpha / \abs{\log\alpha} > \frac{1}{I^\theta} \quad \text{ and } \log m_\alpha = o(\abs{\log\alpha}).
\end{equation*}
Then, under some smoothness conditions \cite{lai1998infobd}, the window-limited GLR-CuSum test:
\begin{equation} \label{eq:lai_glr_test}
    \td{\tau}_{\text{GLR}}\left(b\right) := \inf \left\{n \geq 1:\max_{(n-m_\alpha)^+ < k \leq {n}} \sup_{\theta \in \Theta} \sum_{i=k}^n \log \frac{p_1^\theta(X_i)}{p_0(X_i)} \geq b \right\}
\end{equation}
with test threshold $b_\alpha = \abs{\log\alpha} (1+o(1))$ solves the problem in \eqref{eq:opt_prob_alph} asymptotically as $\alpha \to 0$, for every $\theta \in \Theta$. The asymptotic performance is
\begin{equation}
\label{eq:lai_perf}
    \inf_{\tau \in \mathcal{C}_\alpha} \WADDth{\tau} \sim \WADDth{\td{\tau}_{\text{GLR}} \left(b_\alpha\right)} \sim \frac{\abs{\log \alpha}}{I^\theta}.
\end{equation}
% Note that $I = \KL{p_1}{p_0}$ when the observations are independent in both the pre- and the post-change regimes.
% \vvv{VVV: Actually this is still confusing since we never mentioned anywhere earlier that we are allowing for the observations to be non-i.i.d. So we should simply say that $I = \KL{p_1}{p_0}$.}

\subsection{Non-parametric GLR CuSum Test}

For the case when $p_1$ is unknown, we define the non-parametric GLR statistic as
\begin{equation}
\label{eq:est_llr}
    \widehat{Z}^{n,k}_i = \log \frac{\widehat{p}^{n,k}_{-i}(X_i)}{p_0(X_i)},\ \forall k \leq i \leq n.
\end{equation}
We remind readers of the definition of $\widehat{p}^{n,k}_{-i}$ from Section~\ref{sec:de_property}. The non-parametric generalized likelihood ratio (NGLR) CuSum stopping rule is defined as
\begin{equation}
\label{def:loo_cusum}
    \widehat{\tau}(b) := \inf \left\{n > 1:\max_{(n-m_b)^+ < k \leq n-1} \sum_{i=k}^n \widehat{Z}^{n,k}_i \geq b \right\}.
\end{equation}
Here the window size $m_b$ is designed to satisfy
\begin{equation} \label{eq:mb_lower}
    \liminf m_b / b \geq \frac{\eta}{I}
\end{equation}
where $\eta > 1$ is an arbitrary constant.

In Lemma~\ref{lem:fa_loo_0}, we show that $\widehat{\tau}$ with a properly chosen density estimator and threshold $b = b_\alpha$ satisfies the false alarm constraint asymptotically in \eqref{fa_constraint}.
% In Lemma~\ref{lem:fa_loo}, we provide another way to set the threshold $b = b'_\alpha$ such that the false alarm constraint is satisfied asymptotically given theoretical characterizations of the density estimator and the window size.
% In Lemma~\ref{lem:fa_loo}, we show that $\widehat{\tau}$ with a properly chosen threshold $b_\alpha$ satisfies the false alarm constraint in \eqref{fa_constraint}.
In Lemma~\ref{lem:delay_loo}, we establish an asymptotic upper bound on $\WADD{\widehat{\tau}(b)}$. The proofs of the lemmas are given in the Appendix. Finally, in Theorem~\ref{thm:opt_loo}, we combine the lemmas and establish the first-order asymptotic optimality of the NGLR-CuSum test. 

%%%%%%%%%%%%%%%%%%%%%%%%%%%%%%%%%%%%%%%%%%%%
\begin{lemma} \label{lem:fa_loo_0}
Suppose that the estimator is chosen such that $\exists ~\varsigma > 0$,
\begin{equation} \label{eq:loo_fa_prod_assump}
    \mathbb{E}_\infty\sbrc{\max_{(k,n):k \leq n \leq k+m_b} \prod_{i=k}^{n} \frac{\widehat{p}^{n,k}_{-i}(X_i)}{p_0(X_i)} } \leq b^{\varsigma}
\end{equation}
for any large enough $b$.
Let $b_\alpha$ satisfy
\begin{equation} \label{loo:choose_b_alpha_0}
    b_\alpha - \varsigma \log b_\alpha = \abs{\log \alpha} + \log 8.
\end{equation}
Then,
\begin{equation*}
    \E{\infty}{\widehat{\tau}(b_\alpha)} \geq \alpha^{-1} (1+o(1)).
\end{equation*}
\end{lemma}

We will elaborate on condition~\eqref{eq:loo_fa_prod_assump} in Section~\ref{num:loo_assump_dis}.
Intuitively, this condition is satisfied when the density estimator converges to the true density fast enough ($\mathbb{P}_\infty$-almost surely). Since the pre-change distribution is known, we can numerically verify condition~\eqref{eq:loo_fa_prod_assump} with the chosen window size and density estimator.

\begin{lemma}
\label{lem:delay_loo}
Suppose $b$ is large enough such that condition \eqref{eq:mb_lower} holds. Suppose that \eqref{eq:converg_m1_bound_loo} and \eqref{eq:converg_m2_bound_loo} hold.
Further, suppose \eqref{eq:lai_lower} holds for the true log-likelihood ratio.
% \vvv{Move (18) to right before this lemma if you skip the discussion regarding the parametric uncertainty case in the previous subsection.}
Then,
\begin{equation*}
    \WADD{\widehat{\tau}(b)} \leq \frac{b}{I} (1+o(1)),~\text{as } b \to \infty.
\end{equation*}
\end{lemma}

\begin{theorem}
\label{thm:opt_loo}
Suppose that conditions \eqref{eq:lai_upper} and \eqref{eq:lai_lower} hold for the true log-likelihood ratio, and suppose that the window size satisfies \eqref{eq:mb_lower}.
Suppose that \eqref{eq:loo_fa_prod_assump} is satisfied for the chosen estimator. Let $b_\alpha$ be so selected according to equation~\eqref{loo:choose_b_alpha_0} such that
% $\FAR{\widehat{\tau}(b_\alpha)}\leq \alpha$ or at least
$\FAR{\widehat{\tau}(b_\alpha)}\leq \alpha (1+o(1))$, where also $b_\alpha = \abs{\log\alpha}(1+o(1))$. Then $\widehat{\tau}(b_\alpha)$ solves the problem in \eqref{prob_def} asymptotically as $\alpha \to 0$, and
\begin{equation*}
    \inf_{\tau \in \mathcal{C}_\alpha} \WADD{\tau} \sim \WADD{\widehat{\tau} \left(b_\alpha\right)} \sim \frac{\abs{\log \alpha}}{I}.
\end{equation*}
\end{theorem}

% \begin{corollary}
% Suppose that conditions \eqref{eq:lai_upper} and \eqref{eq:lai_lower} hold for the true log-likelihood ratio, and suppose that the window size satisfies \eqref{eq:mb_lower} and \eqref{eq:mb_upper}. Suppose that the estimator satisfies \eqref{loo:norm_inf_conc} and \eqref{loo:pow_upper_bd}. Let $b'_\alpha$ be selected according to equation~\eqref{loo:choose_b_alpha_0}, and thus $b'_\alpha = \abs{\log\alpha}(1+o(1))$. Then $\widehat{\tau}(b'_\alpha)$ also solves the problem in \eqref{prob_def} asymptotically as $\alpha \to 0$, and
% \begin{equation*}
%     \inf_{\tau \in \mathcal{C}_\alpha} \WADD{\tau} \sim \WADD{\widehat{\tau} \left(b'_\alpha\right)} \sim \frac{\abs{\log \alpha}}{I}.
% \end{equation*}
% \end{corollary}

\begin{proof}[Proof of Theorem~\ref{thm:opt_loo}]
The asymptotic lower bound on the delay follows from \eqref{eq:lai_lower} using \cite[Thm.~1]{lai1998infobd}. The asymptotic optimality of $\widehat{\tau}(b_\alpha)$ follows from Lemma~\ref{lem:fa_loo_0} and Lemma~\ref{lem:delay_loo}.
% The asymptotic optimality of $\widehat{\tau}(b'_\alpha)$ follows from Lemma~\ref{lem:fa_loo} and Lemma~\ref{lem:delay_loo}.
\end{proof}

% \begin{remark}
% All the results above can be naturally extended to the parametric case, where the post-change distribution belongs to a parametric family of distribution and a parametric estimate (e.g., MLE) is employed that satisfies \eqref{eq:converg_m1_bound_loo} and \eqref{eq:converg_m2_bound_loo}.
% \end{remark}

\section{QCD with NWLA CuSum Test}
\label{sec:qcd_adp}

Define the non-parametric window-limited adaptively-estimated log-likelihood ratio as
\begin{equation}
\label{wla:llr}
    \widehat{Z}^w_n = \log \frac{\widehat{p}_n^w(X_n)}{p_0(X_n)},\ \forall n > w
\end{equation}
where $\widehat{p}_n^w$ is the output of the density estimator given input $X^{[n-w,n-1]}$. Note that $\widehat{Z}^w_n$ is independent of ${\cal F}_{n-w-1}$.
Define the non-parametric window-limited adaptive (NWLA) CuSum statistic as:
\begin{equation}
\label{wla:stat}
    \overline{W}^w(n) = \brc{\overline{W}^w(n-1)}^+ + \widehat{Z}^w_n,\quad n > w
\end{equation}
and $\overline{W}^w(1) = \dots = \overline{W}^w(w) = 0$.
The corresponding stopping rule is
\begin{equation}
\label{def:wla_cusum}
    \overline{\tau}(b) := \inf \left\{n > w: \overline{W}^w(n) \geq b \right\}.
\end{equation}
Here $b = b_\alpha > 0$ is a threshold depending on the false alarm rate $\alpha$.
% Also, define the quick-start WLA CuSum statistic as
% \begin{equation}
% \label{wla:stat_quick}
%     \widetilde{W}^w(n) = \left\{
%     \begin{array}{ll}
%       \brc{\widetilde{W}^w(n-1)}^+ + \widehat{Z}^{n-1}_n,   & 2 \leq n \leq w\\
%       \brc{\widetilde{W}^w(n-1)}^+ + \widehat{Z}^w_n,   & n > w
%     \end{array}\right.
% \end{equation}
% and $\widetilde{W}(1) = 0$. The corresponding stopping rule is
% \begin{equation}
% \label{def:wla_cusum_quick}
%     \widetilde{\tau}(b) := \inf \left\{n > 1: \widetilde{W}^w(n) \geq b \right\}.
% \end{equation}
% It can be observed that, for any $b > 0$,
% \begin{equation*}
%     \widetilde{\tau}(b) \leq \overline{\tau}(b),~a.s.
% \end{equation*}
We omit the dependency of $\overline{W}$ on $w$ for brevity.

The following observations regarding condition \eqref{eq:converg_m1_bound} are useful for the analysis in this section. 
If the estimated density $p = p_1$, equation~\eqref{eq:converg_m1_bound} is equivalent to
\begin{equation}
\label{eq:hatI_lower_bound}
    \widehat{I} := \E{1}{\widehat{Z}^w_{n}} \geq I - \frac{C_1}{w^{\beta_1}},\quad n > w
\end{equation}
when $w$ is large. This guarantees that $\widehat{I} > 0$ for all sufficiently large $w$'s.
% Also, for any $w$,
% \[ \widehat{I} = I - \E{1}{\log \frac{p_1(X_n)}{\widehat{p}_n^w(X_n)}} < I \]

In Lemma~\ref{lem:fa_wla}, we show that $\overline{\tau}$ with a properly chosen threshold $b_\alpha$ satisfies the false alarm constraint in \eqref{fa_constraint}. In Lemma~\ref{lem:delay_wla}, we establish an asymptotic upper bound on $\WADD{\overline{\tau}(b_\alpha)}$. Finally, in Theorem~\ref{thm:opt_wla}, we combine the lemmas and establish the first-order asymptotic optimality of the NWLA-CuSum test. It should be mentioned that the results in this section are similar to those in \cite{xie2022windowlimited}, in which a window-limited adaptive CuSum test is studied for the case where there is parametric uncertainty in the post-change regime. However, the results in \cite{xie2022windowlimited} are clearly not applicable to the non-parametric setting studied here. 

The proofs of Lemmas \ref{lem:fa_wla}, \ref{lem:wla_esssup}, and \ref{lem:delay_wla} are given in the Appendix.

\begin{lemma}
\label{lem:fa_wla}
% Suppose that the estimator satisfies 
% \begin{equation}
% \label{wla:lower_rate}
%     \E{\infty}{\KL{p_0}{\widehat{p}_n^w}} \geq \delta_w > 0,~\forall n > w
% \end{equation}
% for some $w > 0$, where $\delta_w \searrow 0$ as $w \to \infty$. Then, with this $w$,
For any $w > 0$,
\begin{equation*}
    \E{\infty}{\overline{\tau}(b)} \geq e^b.
\end{equation*}
Thus, $\overline{\tau}(\overline{b}_\alpha) \in \mathcal{C}_\alpha$ if $\overline{b}_\alpha = \abs{\log \alpha}$.
\end{lemma}

% \begin{remark}
% Condition \eqref{wla:lower_rate} is usually satisfied for KDE. In \cite{hall1987klloss}, it is shown that under some regularity conditions for both $\widehat{p}_n^w$ and $p_0$, when the kernel bandwidth is chosen as some polynomial rate of decay in $w$, the asymptotic KL-loss of KDE is also \textit{equal} to some polynomial rate of decay in $w$.
% \end{remark}

% \begin{corollary}
% Under the same condition as in Lemma~\ref{lem:fa_wla},
% \begin{equation*}
%     \E{\infty}{\widetilde{\tau}(b)} \geq e^b
% \end{equation*}
% for any finite $w$. Thus, $\widetilde{\tau}(\overline{b}_\alpha) \in \mathcal{C}_\alpha$.
% \end{corollary}
% \begin{proof}
% Define another SR-like statistic:
% \begin{equation*}
%     \widetilde{R}_n = \left\{\begin{array}{ll}
%     (1 + \widetilde{R}_{n-1}) e^{\widehat{Z}^{n-1}_n} & 1 < n \leq w \\
%     (1 + \widetilde{R}_{n-1}) e^{\widehat{Z}^w_n} & n > w
% \end{array}\right.
% \end{equation*}
% and $\widetilde{R}_1 = 0$. Define the corresponding stopping rule $\widetilde{\tau}_R(b) := \inf \left\{n: \widetilde{R}_n \geq e^b \right\}$. The rest of the proof follows similarly, and
% \[ \E{\infty}{\widetilde{\tau}(b)} \geq \E{\infty}{\widetilde{\tau}_R(b)} = 1 + \E{\infty}{\widetilde{R}_{\widetilde{\tau}_R}} \geq e^b. \qedhere \]
% \end{proof}

Before introducing the main lemma on the delay, we first introduce three helping lemmas below.

\begin{lemma}
\label{lem:wla_esssup}
For any change-point $\nu \geq 1$ and $b > 0$,
\[\esssup \E{\nu}{(\overline{\tau}(b) - \nu + 1)^+ | {\cal F}_{\nu-1} } \leq \E{1}{\overline{\tau}(b)}. \]
\end{lemma}

\begin{lemma}
\label{lem:dominance}
Define
\[ U_n = U_{n-1} + \widehat{Z}^w_n,\quad \forall n > w \]
with $U_1 = \dots = U_w = 0$. Also define the stopping time
\begin{equation}
\label{def:tau_u}
    \tau_u(b) := \inf \{ n > w: U_n \geq b \}.
\end{equation}
Then, $\tau_u(b) \geq \overline{\tau}(b)$ on $\cbrc{\tau_u(b) < \infty}$ for any $b > 0$.
\end{lemma}
\begin{proof}[Proof of Lemma~\ref{lem:dominance}]
The proof is similar to \cite[Lemma~5]{xie2022windowlimited}. Note that $U_w = 0 = \overline{W}(w)$. For any $k \geq w$, if $U_k \leq \overline{W}(k)$, then
\[ U_{k+1} = U_{k} + \widehat{Z}^w_k \leq \overline{W}(k) + \widehat{Z}^w_k \leq \brc{\overline{W}(k)}^+ + \widehat{Z}^w_k = \overline{W}(k+1) \quad a.s. \]
Thus by induction, $\tau_u(b) \geq \overline{\tau}(b)$ on $\cbrc{\tau_u(b) < \infty}$.
\end{proof}

\begin{lemma}
\label{lem:wla_stop}
If $\widehat{I} > 0$, then the $\tau_u$ defined in \eqref{def:tau_u} satisfies $\tau_u < \infty$ almost surely under $\mathbb{P}_1$.
\end{lemma}
\begin{proof}[Proof of Lemma~\ref{lem:wla_stop}]
First, for any $k >0$,
\[ \sum_{i=w+1}^{w+wk} \widehat{Z}^w_i = \sum_{j=1}^w \sum_{\ell=1}^{k} \widehat{Z}^w_{w \ell + j} \]
and given $j$, $\sum_{\ell=1}^{k} \widehat{Z}^w_{w \ell + j}$ is a sum of i.i.d. random variables under $\mathbb{P}_1$.
Define
\[ \tau_{u,j} := w \cdot \inf\cbrc{ k \geq 1: \sum_{\ell=1}^{k} \widehat{Z}^w_{w \ell + j} \geq b / w} + j. \]
Note that $\tau_{u,j}$ is the boundary crossing time of a sum of i.i.d. random variables with mean $\widehat{I}$. 
Now, if $\widehat{I} = \E{1}{\widehat{Z}^w_i} > 0,~\forall i > w$, then by \cite[Prop.~8.21]{siegmund1985}, $\E{1}{\tau_{u,j}} < \infty,~\forall j=1,\dots,w$.
Finally, since
\[ \forall j,~\sum_{\ell=1}^{k} \widehat{Z}^w_{w \ell + j} \geq b / w \implies \sum_{i=w+1}^{w+wk} \widehat{Z}^w_i \geq b, \]
we have that $\tau_u \leq \max_{j=1,\dots,w} \tau_{u,j}$ almost surely under $\mathbb{P}_1$, and thus $\E{1}{\tau_u} < \infty$. \qedhere

% Thus, for $k > \frac{b}{\eps w}$,
% \begin{align*}
%     \Prob{1}{\sum_{d=1}^{k} \widehat{Z}^w_{j^* + w d} < \frac{b}{w}} &\leq \Prob{1}{\sum_{d=1}^{k} \brc{\widehat{Z}^w_{j^* + w d} - \widehat{I}} < \frac{b}{w} - k \eps}\\
%     &\leq \Prob{1}{\sum_{d=1}^{k} \widehat{Z}^w_{j^* + w d}  < k \brc{\widehat{I} - \eps}}\\
%     &\leq \exp\brc{- k \brc{u (\eps - \widehat{I}) - \log \E{1}{e^{- u \widehat{Z}^w}}}},~\forall u > 0.
% \end{align*}
% Denote the rate function $\Lambda (u) := u (\eps - \widehat{I}) - \log \E{1}{e^{- u \widehat{Z}^w_n}}$. As long as $\E{1}{e^{- u^* \widehat{Z}^w_n}} < \infty$ for some $u^* > 0$, $\Lambda(u)$ is continuously differentiable in $u$ when $u < u^*$, $\Lambda (0) = 0$, $\Lambda'(0) = \eps - \widehat{I} + \E{1}{\widehat{Z}^w_n} = \eps > 0$, and $\Lambda''(u) > 0$ on $u < u^*$. Therefore, we get $\Lambda (u^*) > 0$. Putting these together,
% \[ \Prob{1}{\tau_u > w+kw} \leq \Prob{1}{\sum_{i=w+1}^{w+wk} \widehat{Z}^w_i < b} \leq \Prob{1}{\exists j^*,\sum_{d=1}^{k} \widehat{Z}^w_{w d + j^*} < \frac{b}{w}} \leq e^{-k \Lambda (u^*)}\]
% and finally,
% \begin{align*}
%     \E{1}{\tau_u} = w \E{1}{\tau_u / w} &\leq w + \ceil{\frac{b}{\eps}} + w \sum_{k=\ceil{\frac{b}{\eps w}}}^{\infty} \Prob{1}{\tau_u / w > k+1}\\
%     &\leq w + \ceil{\frac{b}{\eps}} + w \sum_{k=\ceil{\frac{b}{\eps w}}}^{\infty} e^{-k \Lambda (u^*)} < \infty. \qedhere
% \end{align*}
\end{proof}

Now, using the lemmas above, we can upper bound the delay of the NWLA-CuSum test.

\begin{lemma}
\label{lem:delay_wla}
% Under the same conditions as in Lemma~\ref{lem:wla_stop}, 
Suppose that $w$ is sufficiently large such that $\widehat{I} > 0$.
Suppose further that \eqref{eq:converg_m2_bound} holds for the estimator. Then,
\begin{equation}
\label{delay:bd_wla}
    \E{1}{\overline{\tau}(b)} \leq \E{1}{\tau_u(b)} \leq \widehat{I}^{-1} \brc{b + w \widehat{I} + I + \sqrt{2} \frac{C_2}{\widehat{I} w^{\beta_2}} + \brc{\frac{4 C_2}{\widehat{I} w^{\beta_2}} (b+I)}^\frac{1}{2} }
\end{equation}
where $\tau_u (b)$ is defined in \eqref{def:tau_u}.
\end{lemma}

\begin{theorem}
\label{thm:opt_wla}
Suppose that $\overline{b}_\alpha = \abs{\log\alpha}$ and the window size $w_\alpha = \abs{\log \alpha}^\kappa$ for some $0 < \kappa < 1$.
% Suppose the same condition as in Lemma~\ref{lem:fa_wla} holds.
Then, $\overline{\tau}(\overline{b}_\alpha)$ solves the problem in \eqref{prob_def} asymptotically as $\alpha \to 0$, and the delay is upper-bounded as
\begin{equation*}
    % \inf_{\tau \in \mathcal{C}_\alpha} \WADD{\tau} \sim 
    \WADD{\overline{\tau} \left(\overline{b}_\alpha\right)} \leq \frac{\abs{\log \alpha}}{I} \brc{1 + \Theta(\abs{\log \alpha}^{-\rho_\kappa})},
\end{equation*}
where
\begin{equation}
\label{opt:rate_wla}
    \rho_\kappa = \min \left\{\kappa \beta_1, 1-\kappa \right\}.
\end{equation}
\end{theorem}
\begin{proof}[Proof of Theorem~\ref{thm:opt_wla}]
From \eqref{eq:converg_m1_bound} it follows that,
\begin{align*}
\frac{I}{\widehat{I}} = 1 + \frac{I-\widehat{I}}{\widehat{I}} & \leq 1 + \frac{C_1}{\widehat{I} w^{\beta_1}}\\
&\leq 1 + \frac{C_1}{\brc{I - \frac{C_1}{w^{\beta_1}}} w^{\beta_1}} \\
& = 1 + \frac{C_1}{I w^{\beta_1} - C_1}. 
\end{align*}
Given the selected $w = w_\alpha$, $\widehat{I} > 0$ for a sufficiently small $\alpha$. Define $r_\alpha := \frac{\abs{\log \alpha}}{I}$. From \eqref{delay:bd_wla}, if we select $\overline{b}_\alpha = \abs{\log \alpha}$ and $w = \abs{\log \alpha}^{\kappa}$, the scaled average delay (when $\nu=1$) can be upper bounded as:
% \vvv{This is not the worst-case delay. First of all, there is a scale factor, and secondly, this is just the delay when $\nu=1$.}
% \vvv{[This equation is confusing because of the division by $\frac{\abs{\log \alpha}}{I}$ on the LHS. I know you are trying to write it to fit on one line but maybe there is a better way to do it by redefining terms.]}
\begin{align*}
    &r_\alpha^{-1} \E{1}{\overline{\tau}(\overline{b}_\alpha)} \leq r_\alpha^{-1} \E{1}{\tau_u(\overline{b}_\alpha)}\\
    &\leq \frac{I}{\widehat{I}} \brc{1 + \frac{\widehat{I}}{\abs{\log \alpha}^{1-\kappa}} + \frac{I}{\abs{\log \alpha}} + \sqrt{2} \frac{C_2}{\widehat{I} \abs{\log \alpha}^{1+\kappa \beta_2}} + \frac{1}{\abs{\log \alpha}}\brc{\frac{4 C_2}{\widehat{I} \abs{\log \alpha}^{\kappa \beta_2}} (\abs{\log \alpha}+I)}^\frac{1}{2} }\\
    &\leq 1 + \frac{C_1}{I \abs{\log \alpha}^{\kappa \beta_1} - C_1} + \frac{I}{\abs{\log \alpha}^{1-\kappa}} + o\brc{\frac{1}{\abs{\log \alpha}^{1-\kappa}}}\\
    &= 1 + \Theta\brc{\frac{1}{\abs{\log\alpha}^{\rho_\kappa}}}
\end{align*}
% \begin{equation*}
%     \E{1}{\overline{\tau}(\overline{b}_\alpha)} / \frac{\abs{\log \alpha}}{I} \leq \frac{1 + \Theta\brc{\frac{1}{\abs{\log \alpha}^{1-\kappa}}} + \Theta\brc{\frac{1}{\abs{\log \alpha}^{1+\kappa \beta_2}}} + \Theta\brc{\frac{1}{\sqrt{\abs{\log \alpha}^{1+\kappa \beta_2}}}} }{1 + \Theta\brc{\abs{\log \alpha}^{-\kappa \beta_1}} }.
% \end{equation*}
Together with Lemmas~\ref{lem:fa_wla}--\ref{lem:wla_stop}, the result on the asymptotic delay at $\nu=1$ establishes the asymptotic optimality. %Note that a similar proof can be given for $w = \Theta\brc{\abs{\log \alpha}^{\kappa}}$. \vvv{[I'm confused here. Why are you trying to generalize to $w = \Theta\brc{\abs{\log \alpha}^{\kappa}}$?]} \yuchen{Because I want to include those cases, for example, when $w$ is equal to some constant multiplying $\abs{\log \alpha}^\kappa$. Feel free to comment this last sentence out if you think it is unnecessary.} \qedhere
\end{proof}

\begin{remark}
Since $\beta_1 < 1$, we have $\rho_\kappa < \frac{1}{2}$ if $\kappa \in (0,1)$. Also, from \eqref{opt:rate_wla}, to maximize $\rho_\kappa$, one can choose
\[ \rho^* = \max_{\kappa \in (0,1)} \rho_\kappa, \quad \kappa^* = \arg \max_{\kappa \in (0,1)} \rho_\kappa. \]
% If the density estimator is indeed parametric (see the remark of Lemma~\ref{lem:Dest}), we have $\beta_1 = \frac{1}{2}$, which corresponds to a maximum convergence rate of $\frac{1}{2}$. This result aligns with that in \cite[Thm~2]{xie2022windowlimited}.
\end{remark}

\begin{example}
Let the dominating measure $\mu$ be the Lebesgue measure on $\R^d$. Recall the definition of ${\cal H}_\gamma$ in Example~\ref{ex:kde}. Consider $p_0,p_1 \in {\cal H}_\gamma$ with bounded support and non-zero density. Consider using KDE as the estimator with some kernel that satisfies \eqref{ex:kde_kernel_order}. Previously we showed in \eqref{ex:kde_orderd} and Lemma~\ref{lem:Dest} that with a properly chosen $\bm{h}$, the optimal $\beta_1$ and $\beta_2$ are
% \vvv{[You are using the term ``rate of convergence" to mean something different here. Can you drop ``rate of convergence"?]} 
$\beta_1 = \beta_2 = \frac{2 \gamma}{2 \gamma + d}$. Therefore, in this case $\kappa^*$ and $\rho^*$ are
\[ \kappa^* = \arg \max_{\kappa \in (0,1)} \min \left\{\frac{2 \gamma}{2 \gamma + d} \kappa, 1-\kappa \right\} = \frac{2 \gamma + d}{4 \gamma + d}, \]
% and the optimal rate of convergence of WADD is
and
\[ \rho^* = \frac{2 \gamma}{4 \gamma + d}. \]
\end{example}

\section{Numerical Results}
\label{sec:num_res}

In this section, we present some numerical results for the proposed tests. Before presenting the main results, we first investigate from a numerical perspective the feasibility of condition~\eqref{eq:loo_fa_prod_assump}, which is key to Lemma~\ref{lem:fa_loo_0}.

\subsection{Discussion of Condition~\eqref{eq:loo_fa_prod_assump}}
\label{num:loo_assump_dis}

The quantity of interest is
\begin{equation} \label{eq:prod_Q}
    Q(m) := \E{\infty}{\max_{n=2,\dots,m}  \prod_{i=1}^{n} \frac{\widehat{p}^{n,1}_{-i}(X_i)}{p_0(X_i)} }.
\end{equation}
In Fig.~\ref{fig:loo_max_prod}, we study the numerical properties of $Q(m)$ for KDE with Gaussian kernels. The difference $\log(Q(m)) - 3\log m$ is plotted against $m$ for each $m = 5,10,\dots,100$. In the plot, it is observed that 
\begin{equation*}
    \log(Q(m)) - 3\log m < 0 \implies Q(m) \leq m^3
    %,~\text{for large } m.
\end{equation*}
Therefore, with the chosen density estimator, if further the window size is chosen such that $m_b \leq b^\vartheta (1+o(1))$ with some $\vartheta > 1$, then condition~\eqref{eq:loo_fa_prod_assump} is satisfied with $\varsigma = 3 \vartheta$.

\begin{figure}[tbp]
\centering
\includegraphics[width=0.8\textwidth]{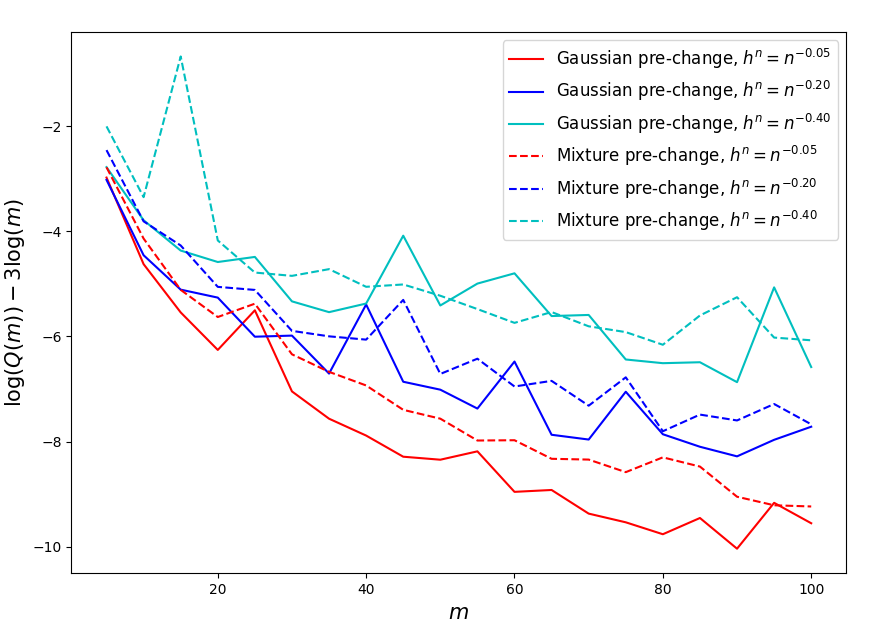}
% \vspace{-3mm}
\caption{$\log(Q(m)) - 3\log m$ versus $m$ with $Q(m)$ defined in \eqref{eq:prod_Q}. For each value of $m$, 100000 Monte Carlo runs are performed for each $n=2,\dots,m$, and the maximum product of likelihood ratios is averaged. The selected pre-change distributions are $\calN(0,1)$ (solid lines) and $\frac{1}{3}(\calN(-2,\frac{1}{4})+\calN(0,\frac{1}{4})+\calN(2,1))$ (dashed lines), with $\calN(\mu,\sigma^2)$ denote a Gaussian with mean $\mu$ and $\sigma^2$. The KDE with Gaussian kernel is used for density estimation, with bandwidth $h = n^{-r}$ where $r$ chosen as 0.05 (red), 0.2 (blue), and 0.4 (cyan). In both plots, it is observed that the difference trends lower as $m$ increases, and that all simulated values are below zero.
}
\label{fig:loo_max_prod}
\end{figure}

% Before presenting Monte Carlo simulation results, we comment on the computational complexity of all the tests mentioned in the following table. We assume that the density estimator of interest is KDE (defined in \eqref{def:kde}).
% \begin{center}
% \begin{tabular}{||l | c||} 
%  \hline
%  Test & Complexity at time $n$\\ [0.5ex] 
%  \hline\hline
%  LOO-CuSum & $\Theta((n \wedge m_\alpha)^3)$ \\ 
%  \hline
%  WLA-CuSum & $\Theta(w)$ when $n > w$, 0 otherwise \\
%  \hline
%  parallel-WLA-CuSum & $\Theta((n \wedge W_{\text{max}})^2)$ \\
%  \hline
%  LOO-WLA-CuSum & $\Theta(w)$ when $n > w$, $\Theta(n^3)$ otherwise \\ [1ex] 
%  \hline
% \end{tabular}
% \end{center}

\subsection{Performance of NGLR-CuSum Test}

In Fig.~\ref{fig:loo_perf}, we study the performance of the proposed NGLR-CuSum test (defined in \eqref{def:loo_cusum}) through Monte Carlo (MC) simulations when the pre-change distribution is ${\cal N}(0,1)$. The KDE (defined in \eqref{def:kde}) with a Gaussian kernel is used to estimate the density. The actual post-change distribution is ${\cal N}(0.5,1)$, but this knowledge is not used in the NGLR-CuSum test. The performance of the NGLR-CuSum test is compared against that of the following tests:
\begin{enumerate}
    \item the CuSum test (in \eqref{def:cusum}), which has full knowledge of the post-change distribution;
    \item the parametric window-limited GLR-CuSum test (in \eqref{eq:lai_glr_test}), in which it is assumed that the post-change distribution belongs to $\{{\cal N}(\theta,1)\}_{\theta \neq 0}$.
\end{enumerate}
The change-point is taken to be $\nu=1$.\footnote{Note that $\nu=1$ may not necessarily be the worst-case value for the change-point for the NGLR-CuSum test in general. However, extensive experimentation on this Gaussian mean-change problem with different values of $\nu$ ranging from 1 to 100, with a window-size of 100, shows that $\nu=1$ results in the largest expected delay among all $\nu$'s considered.}
Different window sizes are considered, among which the window size of 100 is sufficiently large to cover the full range of delay.
It is seen that the expected delay of the NGLR-CuSum test is close to that of the GLR-CuSum test for all window sizes considered. 
% The results also validate the first-order asymptotic optimality of the NGLR-CuSum test for large enough window size (Thm~\ref{thm:opt_loo}).

\begin{figure}[tbp]
\centerline{\includegraphics[width=.8\textwidth]{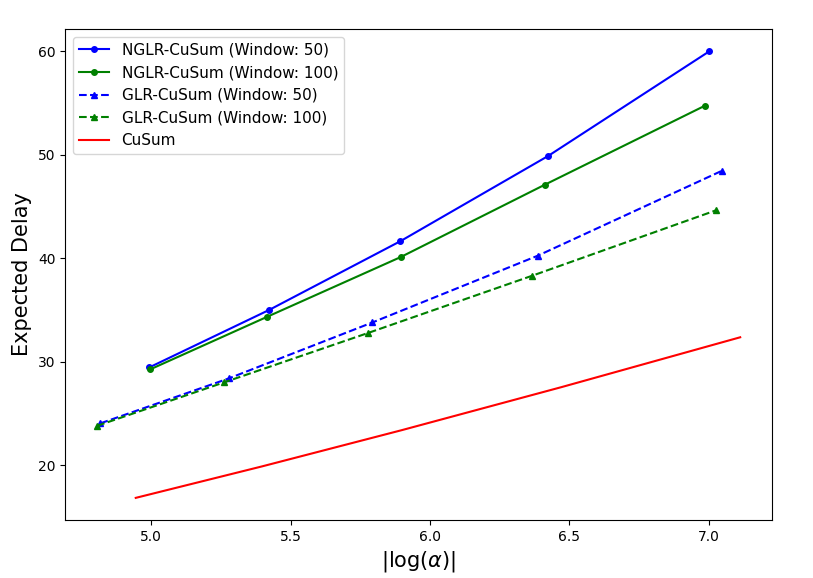}}
\vspace{-3mm}\caption{Comparison of operating characteristics of the NGLR-CuSum test (solid lines) with the CuSum test (in red) and the parametric window-limited GLR-CuSum test (dashed lines) in detecting a shift in the mean of a Gaussian. The pre- and post-change distributions are ${\cal N}(0,1)$ and ${\cal N}(0.5,1)$. The change-point $\nu = 1$. The kernel width parameter $h=10^{-1/5}$.
}
\label{fig:loo_perf}
\end{figure}

\subsection{Performance of NWLA-CuSum Test}

In Fig.~\ref{fig:wla_perf}, we study the performance of the proposed NWLA-CuSum test (defined in \eqref{def:wla_cusum}) through Monte Carlo (MC) simulations when the pre-change distribution is ${\cal N}(0,1)$. The KDE (defined in \eqref{def:kde}) is used to estimate the density. The actual post-change distribution is ${\cal N}(0.5,1)$. This knowledge is not used in the NWLA-CuSum test. The performance of the NWLA-CuSum test is compared against that of the following tests:
\begin{enumerate}
    \item the CuSum test (in \eqref{def:cusum}), which assumes full knowledge of the post-change distribution;
    \item the parametric window-limited GLR-CuSum test (in \eqref{eq:lai_glr_test}), in which it is assumed that the post-change distribution belongs to $\{{\cal N}(\theta,1)\}_{\theta \neq 0}$.;
    \item the parallel-NWLA-CuSum test, defined as
    \begin{equation*}
        \overline{\tau}_{\text{parallel}}(b,W_{\text{max}}) := \inf \left\{n > 1: \max_{1 \leq w \leq W_{\text{max}}} \overline{W}^w(n) \geq b \right\}.
    \end{equation*}
    % Here $\overline{\tau}_{\text{parallel}}(b,W_{\text{max}})$ can start right after the first sample, and the false alarm constraint is satisfied with $b = \log(W_{\text{max}} / \alpha)$.
\end{enumerate}
Using a similar analysis as in Section~\ref{sec:qcd_adp}, it can be shown that the parallel-NWLA-CuSum test is also asymptotically optimal with the threshold chosen as $b_\alpha = \abs{\log \alpha} + \log W_{\text{max}}$. The change-point is taken to be $\nu=1$, which corresponds to the worst-case expected delay for the NWLA-CuSum test (shown in Lemma~\ref{lem:wla_esssup}), the parallel NWLA-CuSum test, and the CuSum test, but not necessarily for the parametric window-limited GLR-CuSum test.
Different window sizes are also considered.
We note that there is a trade-off to consider in the design of the window size for the NWLA-CuSum test. If the window size is too small, the post-change density might not be accurately estimated. On the other hand, if the window size is too large, the test might wait too long before its statistic starts to grow in the post-change regime. To address this trade-off, the parallel-NWLA-CuSum test could be employed without specifying a pre-defined window size, albeit at the expense of having to run more tests in parallel.

\begin{figure}[tbp]
\centerline{\includegraphics[width=.8\textwidth]{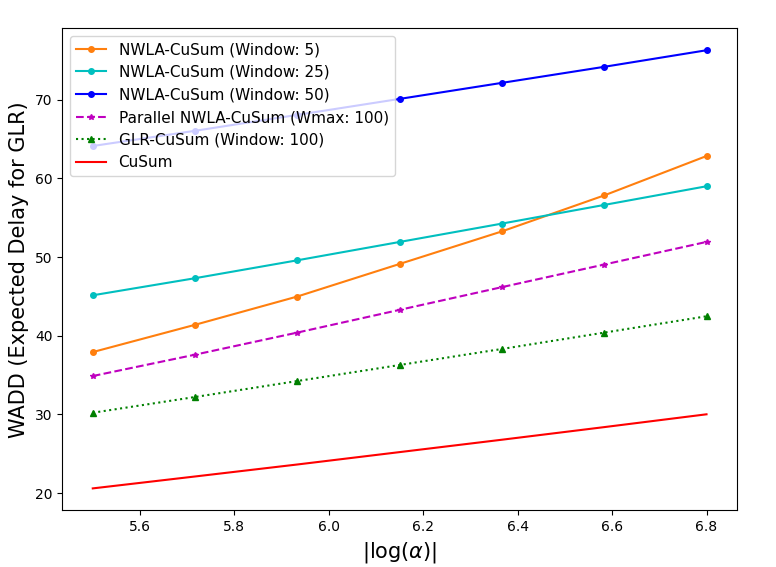}}
\vspace{-3mm}\caption{Comparison of operating characteristics of the NWLA-CuSum test (solid lines) and the parallel-NWLA-CuSum test (dashed line) with the CuSum test (in red) and the parametric window-limited GLR-CuSum test (dotted line) in detecting a shift in the mean of a Gaussian. The pre- and post-change distributions are ${\cal N}(0,1)$ and ${\cal N}(0.5,1)$. The change-point $\nu = 1$. The kernel width parameter $h=w^{-1/5}$, where $w$ is the window size.
}
\label{fig:wla_perf}
\end{figure}

\subsection{Comparison between NGLR-CuSum Test and NWLA-CuSum Test}

We now compare the performance between the NGLR-CuSum test and the parallel-NWLA-CuSum test. First, we compare the computational complexity of both tests if the KDE is used for density estimation. For the NGLR-CuSum test, at each input observation $X_n$, for all hypothesized change-points $k \in ((n-m)^+,n-1]$ (where $m$ is the window size), the pair-wise kernel value $K(X_i,X_j)$ for each pair of $i,j \in [k,n]$ with $i \neq j$ is calculated and a LOO kernel estimate is %formulated 
evaluated at each point $X_k,\dots,X_n$. Then, the LOO-estimated log-likelihood ratios at these points are summed up and the maximum sum (over $k$) is compared to the given threshold. Thus, the computation complexity of the NGLR-CuSum test at each time is $\Theta((n \wedge m)^3)$. For the parallel-NWLA-CuSum test, at each given $X_n$, $K(X_i,X_n)$ is first calculated for each $i \in [(n-W_{\text{max}}) \vee 1,n-1]$. Then, with each possible window size $w=1,\dots,W_{\text{max}} \wedge (n-1)$, a WLA kernel estimate is evaluated to update the corresponding NWLA-CuSum statistic in an efficient manner. Finally, the maximum NWLA-CuSum statistic (over $w$) is compared to the given threshold. Therefore, the computation complexity of the parallel-NWLA-CuSum test is $\Theta((n \wedge W_{\text{max}})^2)$. Also note that in practice, $W_{\text{max}}$ is usually chosen to be smaller than $m$, by comparing the requirement of $m_b$ in \eqref{eq:mb_lower} with the window condition in Thm~\ref{thm:opt_wla}.

In Fig.~\ref{fig:loo_vs_pwla}, we compare the numerical performance of the NGLR-CuSum test and the parallel-NWLA-CuSum test in detecting a shift in the mean of a Gaussian. When the window size is large enough, both tests achieve similar performance at $\nu=1$, which corresponds to the worst-case $\nu$ for the parallel-NWLA-CuSum test, but not necessarily for the NGLR-CuSum test.

\begin{figure}[tbp]
\centerline{\includegraphics[width=.8\textwidth]{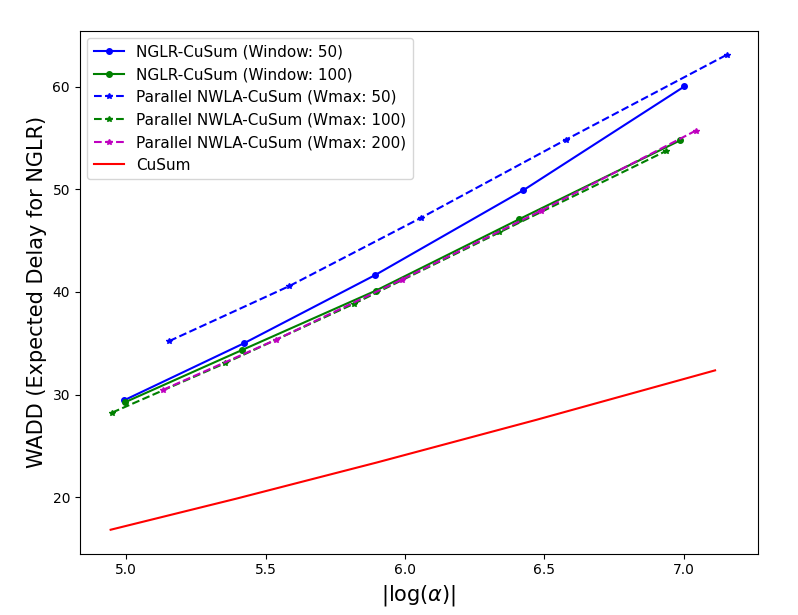}}
\vspace{-3mm}\caption{Comparison of operating characteristics of the NGLR-CuSum test (solid lines) and the parallel-NWLA-CuSum test (dashed lines) with the CuSum test (in red) in detecting a shift in the mean of a Gaussian. The pre- and post-change distributions are ${\cal N}(0,1)$ and ${\cal N}(0.5,1)$. The change-point $\nu = 1$. The kernel width parameters are $h=10^{-1/5}$ for the NGLR-CuSum test, and $h=w^{-1/5}$ for the parallel-NWLA-CuSum test, where $w$ is the window size.
}
\label{fig:loo_vs_pwla}
\end{figure}

\section{Conclusion}
\label{sec:concl}

We studied a window-limited non-parametric generalized likelihood ratio (NGLR) CuSum test and a non-parametric window-limited adaptive (NWLA) CuSum test for QCD. Both tests do not assume any explicit knowledge of the post-change distribution, and do not require post-change training samples ahead of time. We characterized a generic class of density estimators that enable detection. For both tests, we provided a way to set the test thresholds to meet false alarm constraints, and we showed that the tests are first-order asymptotically optimal with the selected thresholds, as the false alarm rate goes to zero. 
% Further, we characterized the convergence rate in the delay of the NWLA-CuSum test as a function of the false alarm rate \vvv{[I'm still confused by this sentence. Can you explain what you mean by convergence rate?]}. 
We validated our analysis through Monte-Carlo simulations, in which we compared both tests with baseline tests that have distributional knowledge. 
% \vvv{Any directions for future research?}

%An interesting direction for future research is the study of quickest change detection problems where neither the pre-change nor the post-change distribution is known, but a set of pre-change samples is available . It is also interesting to apply recent deep learning density estimation methods to quickest change detectors.

\section*{Acknowledgements}
The authors would like to thank George Moustakides for helpful discussions regarding the NWLA-CuSum test.

% \vvv{The references need updating. For example, references 5 and 6 do not have volume, number, month, page information.} \yuchen{Fixed.}

\bibliographystyle{IEEEtran}
\bibliography{ref}

\appendix

\begin{proof}[Proof of Lemma~\ref{lem:Dest}]
For brevity we write $\widehat{p}(X) = \widehat{p}^{w}_n(X_n)$, and note that $X$ is independent of $\widehat{p}$. We use the fact that $\log{s} \leq (s-1)$ to establish an upper bound on the first moment.
In particular,
\begin{align} \label{eq:mise_est_cond_m1_bound_proof}
    \Ep{\log\frac{p(X)}{\widehat{p}(X)}} &\leq \Ep{\frac{p(X)}{\widehat{p}(X)} - 1} \nonumber\\
    &= \Ep{\int \frac{p^2(x) - p(x) \widehat{p}(x)}{\widehat{p}(x)} d \mu(x)} \nonumber\\
    &\stackrel{(*)}{=} \Ep{\int \frac{p^2(x) - 2 p(x) \widehat{p}(x) + \widehat{p}^2(x)}{\widehat{p}(x)} d \mu(x)} \nonumber\\
    &\leq \frac{1}{\underline{\zeta}} \MISE(p, \widehat{p})
\end{align}
where $(*)$ follows by the independence between $\widehat{p}$ and $X$ and because both $p$ and $\widehat{p}$ are densities. This establishes \eqref{eq:converg_m1_bound} with $\beta_1 = \beta_3$. 
The proof for \eqref{eq:converg_m1_bound_loo} is similar, noting the independence between $\widehat{p}^{n,k}_{-i}$ and $X_i$.
% Similarly,
% \[ \Ep{\log\frac{p(X_i)}{\widehat{p}_i(X_i)}} \geq \frac{\log(\overline{\zeta}/ \underline{\zeta})}{1-\underline{\zeta}/ \overline{\zeta}} \Ep{\frac{p(X_i)}{\widehat{p}_i(X_i)} - 1} \geq \frac{\log(\overline{\zeta}/ \underline{\zeta})}{\overline{\zeta} - \underline{\zeta}} \MISE(p, \widehat{p}^{n,k}_{-i}). \]
For the second moment, note that $(\log{s})^2 \leq r (s-1)^2$ on $s \geq \underline{\zeta} / \overline{\zeta}$ with $r$ as defined in \eqref{eq:rdef}. Thus,
\begin{align} \label{eq:mise_est_cond_m2_bound_proof}
    \Ep{\brc{\log\frac{p(X)}{\widehat{p}(X)}}^2} &\leq \Ep{r\left(\frac{p(X)}{\widehat{p}(X)} - 1\right)^2} \nonumber\\
    &= r \Ep{\int \frac{\left(p(x) - \widehat{p}(x)\right)^2}{\widehat{p}^2(x)} p(x) d \mu(x)} \nonumber\\
    &\leq \frac{\overline{\zeta} r}{\underline{\zeta}^2} \MISE (p, \widehat{p})
\end{align}
which shows \eqref{eq:converg_m2_bound} with $\beta_2 = \beta_3$. Furthermore, for \eqref{eq:converg_m2_bound_loo},
\begin{align} \label{eq:mise_est_cond_m2_bound_loo_proof}
    \Ep{\brc{\frac{1}{n-k+1}\sum_{i=k}^n \log\frac{p(X_i)}{\widehat{p}^{n,k}_{-i}(X_i)}}^2} \nonumber
    &\stackrel{(a)}{\leq} \Ep{\frac{1}{n-k+1}\sum_{i=k}^n \brc{\log\frac{p(X_i)}{\widehat{p}^{n,k}_{-i}(X_i)}}^2} \nonumber\\
    &\stackrel{(b)}{=} \Ep{\brc{\log\frac{p(X_n)}{\widehat{p}^{n,k}_{-n}(X_n)}}^2} \nonumber\\
    &\leq \frac{\overline{\zeta} r C_3 }{\underline{\zeta}^2 (n-k+1)^{\beta_3}}.
\end{align}
Here $(a)$ follows by Jensen's inequality, and $(b)$ follows because $\log\frac{p(X_i)}{\widehat{p}^{n,k}_{-i}(X_i)}$ has the same distribution for all $i \in [k, n]$. The proof is now complete. \qedhere
\end{proof}

\begin{proof}[Proof of Corollary~\ref{cor:lem_Dest}]
Following the argument in \eqref{eq:mise_est_cond_m1_bound_proof}, we have
\[ \Ep{\log\frac{p(X)}{\widehat{p}(X)}} = O \brc{\frac{1}{\underline{\zeta}} \MISE(p, \widehat{p})} = O(w^{-(\beta_3-\underline{\beta})}), \]
and thus the first moment results (i.e., that of $\beta_1$) follow immediately for \eqref{eq:converg_m1_bound} and \eqref{eq:converg_m1_bound_loo}.

Now we turn to the second moment. From the definition of $r$,
\[ r = \brc{\frac{\log (\overline{\zeta}/ \underline{\zeta}) }{1 - (\underline{\zeta} / \overline{\zeta})}}^2 \leq (\log (\overline{\zeta}/ \underline{\zeta}))^2 = (\overline{\beta} - \underline{\beta})^2 (\log w)^2 \]
Therefore, following the argument in \eqref{eq:mise_est_cond_m2_bound_proof}, we get
\begin{align*}
    \Ep{\brc{\log\frac{p(X)}{\widehat{p}(X)}}^2} &= O\brc{\frac{\overline{\zeta} r}{\underline{\zeta}^2} \MISE (p, \widehat{p})}\\
    &= O\brc{\frac{w^{\overline{\beta}}}{w^{-2 \underline{\beta}} w^{\beta_3}} (\log w)^2 } \\
    &= O\brc{w^{-(\beta_3 - 2 \underline{\beta} - \overline{\beta} - \varrho)} }
\end{align*}
where $\varrho > 0$ is an arbitrarily small constant.
This shows the second moment result for \eqref{eq:converg_m2_bound}. The second moment result for \eqref{eq:converg_m2_bound_loo} is similar if we follow the argument in \eqref{eq:mise_est_cond_m2_bound_loo_proof}.
% \vvv{The $\Theta$ notation has not been defined. Also, it is not clear how this equation gives the equation for $\beta_2$ in the Corollary. You need to add the steps.}
\end{proof}

\begin{proof}[Proof of Lemma~\ref{lem:fa_loo_0}]
Fix $\ell > 1$. For all thresholds $b > 0$,
\begin{align} \label{eq:loo_fa_prob_window}
    &\Prob{\infty}{\ell \leq \widehat{\tau}(b) < \ell + m_b } \nonumber\\
    &\stackrel{(i)}{=} \Prob{\infty}{\exists (k,n) \text{ with } \ell \leq n < \ell + m_b, (n-m_b)^+ < k \leq n-1 : \sum_{i=k}^n \widehat{Z}^{n,k}_i \geq b } \nonumber\\
    &\stackrel{(ii)}{\leq} \Prob{\infty}{\exists (k,n) \text{ with } (\ell-m_b)^+ < k < \ell + m_b, k+1 \leq n \leq k+m_b : \sum_{i=k}^n \widehat{Z}^{n,k}_i \geq b } \nonumber\\
    &= \mathbb P_\infty \brc{\bigcup_{k = (\ell-m_b)^++1}^{\ell+m_b-1} \cbrc{ \exists n \text{ with } k + 1 \leq n \leq k + m_b: \sum_{i=k}^n \widehat{Z}^{n,k}_i \geq b } } \nonumber\\
    &\leq \sum_{k = (\ell-m_b)^++1}^{\ell+m_b-1} \Prob{\infty}{\exists n \text{ with } k + 1 \leq n \leq k + m_b: \sum_{i=k}^n \widehat{Z}^{n,k}_i \geq b } \nonumber\\
    &\leq \sum_{k=(\ell-m_b)^++1}^{\ell+m_b-1} \Prob{\infty}{\tau_k(b) \leq k + m_b}
\end{align}
where $(i)$ follows from the definition of $\widehat{\tau}(b)$, and $(ii)$ follows because
\[ \ell \leq n < \ell + m_b, ~~(n-m_b)^+ < k \leq n-1 
\]
implies that
\[
(\ell-m_b)^+ < k < \ell + m_b,~~ k+1 \leq n \leq k+m_b. 
\]
Here we define the auxiliary stopping time $\tau_k(b)$ for $k \geq 1$ as
\begin{equation}
    \tau_k(b) := \inf\left\{n \in [k + 1, k + m_b]: \sum_{i=k}^n \widehat{Z}^{n,k}_i \geq b \right\}
\end{equation}
and we define $\inf\emptyset := \infty$. Now, for each $k \in [(\ell-m_b)^+, \ell+m_b)$, we have
\begin{align} \label{eq:loo_fa_prob_tauk}
    &\Prob{\infty}{k+1 \leq \tau_k(b) \leq k + m_b} \nonumber\\
    &= \int \ind{k+1 \leq \tau_k(b) \leq k + m_b} d \mathbb{P}_\infty \nonumber\\
    &= \int \ind{k+1 \leq \tau_k(b) \leq k + m_b} \prod_{i=k}^{\tau_k(b)} \frac{\widehat{p}^{\tau_k(b),k}_{-i}(x_i)}{p_0(x_i)} \prod_{i=k}^{\tau_k(b)} \frac{p_0(x_i)}{\widehat{p}^{\tau_k(b),k}_{-i}(x_i)} d \mathbb{P}_\infty \nonumber\\
    &\stackrel{(iii)}{\leq} e^{-b} \int \ind{k+1 \leq \tau_k(b) \leq k + m_b} \prod_{i=k}^{\tau_k(b)} \frac{\widehat{p}^{\tau_k(b),k}_{-i}(x_i)}{p_0(x_i)} d \mathbb{P}_\infty
\end{align}
where $(iii)$ follows from the definition of $\tau_k(b)$.

Now,
\begin{align*}
    &\int \ind{k+1 \leq \tau_k(b) \leq k + m_b} \prod_{i=k}^{\tau_k(b)} \frac{\widehat{p}^{\tau_k(b),k}_{-i}(x_i)}{p_0(x_i)} d \mathbb{P}_\infty\\
    &\leq \int \max_{n \in [k, k+m_b]} \prod_{i=k}^{n} \frac{\widehat{p}^{n,k}_{-i}(x_i)}{p_0(x_i)} d \mathbb{P}_\infty\\
    &\stackrel{(iv)}{\leq} b^\varsigma (1+o(1))
\end{align*}
where $(iv)$ follows from condition~\eqref{eq:loo_fa_prod_assump}. Combining with \eqref{eq:loo_fa_prob_window} and \eqref{eq:loo_fa_prob_tauk}, we have
\[ \sup_{\ell > 1} \Prob{\infty}{\ell \leq \widehat{\tau}(b) < \ell + m_b } \leq 2 m_b e^{-b} b^\varsigma (1+o(1)), \]
and by \cite[Lemma~2.2(ii)]{tartakovsky2020qcdbook},
\[ \E{\infty}{\widehat{\tau}(b)} \geq \frac{1}{8} e^b b^{-\varsigma} (1+o(1)). \]
Choosing $b = b_\alpha$ then satisfies the false alarm constraint asymptotically. \qedhere

\end{proof}

\begin{proof}[Proof of Lemma~\ref{lem:delay_loo}]
Recall that $I = \KL{p_1}{p_0}$. Define a function $\delta_b$ such that $\delta_0 := 1-\eta^{-1} < 1$, that $\delta_b \in (0,\delta_0)$ is decreasing in $b$, and that $\delta_b \searrow 0$ as $b \to \infty$. Define
\begin{equation}
\label{eq:delay_loo_nb}
    n_b := \floor{\frac{b}{I(1-\delta_b)}}
\end{equation}
% By definition of the window size $m_\alpha$ in \eqref{eq:malpha}, if $b_\alpha = \abs{\log\alpha}(1+o(1))$, then
and thus
\begin{equation*}
    n_b < \frac{b}{I(1-\delta_0)} = \frac{\eta b}{I} \leq m_b
\end{equation*}
when $b$ is large enough. If for now that we can get a large enough $b$ to satisfy
% \begin{equation*}
%     \sup_{t \geq \nu \geq 1} \esssup \Prob{\nu}{\left. \sum_{i=t}^{t+n_b} \widehat{Z}^{t+n_b,t}_i \leq (1-\delta_b) I^{-1} n_b\right| {\cal F}_{t-1}} < \delta_b
% \end{equation*}
% or equivalently,
\begin{equation}
\label{eq:delay_loo_main}
    \Prob{\nu}{\sum_{i=n}^{n+n_b-1} \widehat{Z}^{n+n_b-1,n}_i < b} < 2 \delta_b^2,~\forall (\nu,n): n \geq \nu \geq 1.
\end{equation}
Then in the following, we will show by induction that
\begin{equation} \label{eq:delay_loo_induction}
    \esssup \Prob{\nu}{\widehat{\tau}(b)-\nu+1 > k n_b | \widehat{\tau}(b)-\nu+1 > (k-\ell) n_b, {\cal F}_{\nu-1} } \leq (2 \delta_b^2)^\ell,~\forall \nu \geq 1, k \geq \ell
\end{equation}
when $b$ is large enough.

We will induct on the variable $\ell$. The base case is where $\ell=1$, and we get, $\forall k \geq 1$,
\begin{align}
\label{eq:delay_loo_single_step}
    &\esssup \Prob{\nu}{\widehat{\tau}(b)-\nu+1 > k n_b | \widehat{\tau}(b)-\nu+1 > (k-1) n_b, {\cal F}_{\nu-1} } \nonumber\\
    &\stackrel{(i)}{\leq} \esssup \Prob{\nu}{\left.\widehat{\tau}(b)-\nu+1 > k n_b \right| {\cal F}_{\nu+(k-1)n_b-1} } \nonumber\\
    &\stackrel{(ii)}{\leq} \esssup \Prob{\nu}{\left.\sum_{i=\nu+(k-1)n_b}^{\nu+k n_b - 1} \widehat{Z}^{\nu+k n_b-1, \nu+(k-1)n_b}_i < b \right| {\cal F}_{\nu+(k-1)n_b-1} } \nonumber\\
    &\stackrel{(iii)}{=} \Prob{\nu}{\sum_{i=\nu+(k-1)n_b}^{\nu+k n_b - 1} \widehat{Z}^{\nu+k n_b-1, \nu+(k-1)n_b}_i < b } \nonumber\\
    &\stackrel{(iv)}{\leq} 2 \delta_b^2.
\end{align}
In the series of inequalities above, $(i)$ is by definition of essential supremum and $\widehat{\tau}(b)$, $(iii)$ follows from independence between the event $\Big\{\sum_{i=\nu+(k-1)n_b}^{\nu+k n_b - 1} \widehat{Z}^{\nu+k n_b-1, \nu+(k-1)n_b}_i < b\Big\}$ and ${\cal F}_{\nu+(k-1)n_b-1}$, and $(iv)$ follows from \eqref{eq:delay_loo_main}. The reason for $(ii)$ is as follows.
The event $\{ \widehat{\tau}(b)-\nu+1 > k n_b \}$ implies that no change has been detected until time $n=k n_b + \nu - 1$. In particular, this means that at time $n=k n_b + \nu - 1$,
\[ \max_{(\nu+k n_b-1-m_b)^+ < \kappa \leq \nu+k n_b-2} \sum_{i=\kappa}^{\nu+k n_b-1} \widehat{Z}^{\nu+k n_b-1,\kappa}_i < b. \]
Now, since $n_b \leq m_b$,
\[ \sum_{i=\nu+(k-1)n_b}^{\nu+k n_b - 1} \widehat{Z}^{\nu+k n_b-1, \nu+(k-1)n_b}_i \leq \max_{(\nu+k n_b - 1-m_b)^+ < \kappa \leq \nu+k n_b - 2} \sum_{i=\kappa}^{\nu+k n_b-1} \widehat{Z}^{\nu+k n_b-1,\kappa}_i < b.\]
The induction base is thus established.

% \vvv{It's still not clear to me what you are trying to prove using induction. Why don't you state that first and say that you will be proving the statement using induction? Also, the induction argument needs to be cleanly written, with the induction hypothesis stated clearly first, and what that implies coming next.}

We now turn to the induction step. Suppose we have proved that
\[ \esssup \Prob{\nu}{\widehat{\tau}(b)-\nu+1 > k n_b | \widehat{\tau}(b)-\nu+1 > (k-\ell) n_b, {\cal F}_{\nu-1} } \leq (2 \delta_b^2)^\ell,~\forall k \geq l. \]
Then, for $k \geq \ell + 1$,
\begin{align*}
    &\Prob{\nu}{\widehat{\tau}(b)-\nu+1 > k n_b | \widehat{\tau}(b)-\nu+1 > (k-\ell-1) n_b, {\cal F}_{\nu-1} }\\
    &\stackrel{(v)}{=}\Prob{\nu}{\widehat{\tau}(b)-\nu+1 > k n_b, \widehat{\tau}(b)-\nu+1 > (k-1) n_b | \widehat{\tau}(b)-\nu+1 > (k-\ell-1) n_b, {\cal F}_{\nu-1} }\\
    &= \Prob{\nu}{\widehat{\tau}(b)-\nu+1 > (k-1) n_b | \widehat{\tau}(b)-\nu+1 > (k-\ell-1) n_b, {\cal F}_{\nu-1} } \\
    &\quad \times \Prob{\nu}{\widehat{\tau}(b)-\nu+1 > k n_b | \widehat{\tau}(b)-\nu+1 > (k-1) n_b, \widehat{\tau}(b)-\nu+1 > (k-\ell-1) n_b, {\cal F}_{\nu-1} }
\end{align*}
where $(v)$ holds because $\{\widehat{\tau}(b)-\nu+1 > k n_b\} \subseteq \{\widehat{\tau}(b)-\nu+1 > (k-1) n_b\}$. Thus,
\begin{align*}
    &\esssup \Prob{\nu}{\widehat{\tau}(b)-\nu+1 > k n_b | \widehat{\tau}(b)-\nu+1 > (k-\ell-1) n_b, {\cal F}_{\nu-1} } \\
    &\leq \esssup \Prob{\nu}{\widehat{\tau}(b)-\nu+1 > (k-1) n_b | \widehat{\tau}(b)-\nu+1 > (k-\ell-1) n_b, {\cal F}_{\nu-1} } \\
    &\quad \times \esssup \Prob{\nu}{\widehat{\tau}(b)-\nu+1 > k n_b | \widehat{\tau}(b)-\nu+1 > (k-1) n_b, \widehat{\tau}(b)-\nu+1 > (k-\ell-1) n_b, {\cal F}_{\nu-1} }\\
    &\stackrel{(vi)}{\leq} \esssup \Prob{\nu}{\widehat{\tau}(b)-\nu+1 > (k-1) n_b | \widehat{\tau}(b)-\nu+1 > (k-\ell-1) n_b, {\cal F}_{\nu-1} } \\
    &\quad \times \esssup \Prob{\nu}{\widehat{\tau}(b)-\nu+1 > k n_b | {\cal F}_{\nu+(k-1)n_b-1} }\\
    &\stackrel{(vii)}{\leq} \esssup \Prob{\nu}{\widehat{\tau}(b)-\nu+1 > (k-1) n_b | \widehat{\tau}(b)-\nu+1 > (k-\ell-1) n_b, {\cal F}_{\nu-1} } \times (2 \delta_b^2)\\
    &\leq (2 \delta_b^2)^{\ell+1}
\end{align*}
where $(vi)$ follows by definition of essential supremum and the fact that 
\[
\{\widehat{\tau}(b)-\nu+1 > (k-1) n_b, \widehat{\tau}(b)-\nu+1 > (k-\ell-1) n_b, {\cal F}_{\nu-1}\} \subset \{{\cal F}_{\nu+(k-1)n_b-1}\}
\]
and $(vii)$ follows by \eqref{eq:delay_loo_single_step}.
Therefore, by induction, we get \eqref{eq:delay_loo_induction}. In particular, letting $\ell = k$, we get
\[ \esssup \Prob{\nu}{\widehat{\tau}(b)-\nu+1 > k n_b | {\cal F}_{\nu-1} } \leq (2 \delta_b^2)^k,~\forall \nu \geq 1 \]
for all sufficiently large $b$'s.

% \vvv{what recursion? Again you seem to be skipping steps that are essential.}
% \[ \esssup \Prob{\nu}{\widehat{\tau}(b)-\nu+1 > k n_b | {\cal F}_{\nu-1} } \leq (2 \delta_b^{2})^k.\]

% \begin{align}
%     &\quad \esssup \Prob{\nu}{\widehat{\tau}(b)-\nu+1 > k n_b | {\cal F}_{\nu-1} } \nonumber\\
%     &\leq \esssup \Prob{\nu}{\widehat{\tau}(b)-\nu+1 > (k-1) n_b | {\cal F}_{\nu-1} } \times \nonumber\\ 
%     &\esssup \Prob{\nu}{\widehat{\tau}(b)-\nu+1 > k n_b | \widehat{\tau}(b)-\nu+1 > (k-1) n_b, {\cal F}_{\nu-1} } \nonumber\\
%     &\leq (2 \delta_b^{2})^k.
% \end{align}

Therefore, for all sufficiently large $b$'s,
\begin{align*}
    &\sup_{\nu \geq 1} \esssup \E{\nu}{n_b^{-1} (\widehat{\tau}(b)-\nu+1)^+|{\cal F}_{\nu-1}} \nonumber\\
    &\leq \sum_{k=1}^\infty \esssup \Prob{\nu}{\widehat{\tau}(b)-\nu+1 > k n_b | {\cal F}_{\nu-1}} \nonumber\\
    &\leq \sum_{k=0}^\infty (2 \delta_b^{2})^k = \frac{1}{1-2 \delta_b^{2}}.
\end{align*}
Recall the definition of WADD in \eqref{eq:def_wadd}. As $b \to \infty$, this implies that
\begin{equation*}
    \WADD{\widehat{\tau}(b)} \leq \frac{n_b}{1-2 \delta_b^2} \leq \frac{b}{I (1-\delta_b) (1-2 \delta_b^2)} = \frac{b}{I} (1+o(1)).
\end{equation*}

It remains to show \eqref{eq:delay_loo_main}. Write 
\[ Z_i := \log \frac{p_1(X_i)}{p_0(X_i)}. \]
% Note that ${\cal F}_{t-1}$ can be dropped by independence between ${\cal F}_{t-1}$ and $\widehat{Z}^{t+n_b-1,t}_i$. \vvv{If ${\cal F}_{t-1}$ can be dropped, then why have it in \eqref{eq:delay_loo_main} at all? It is confusing. Better to give the assumption without the conditioning, and explain in step (iii) above that you can drop the conditioning and apply the assumption. Also, you should remove the sup over $t \geq \nu \geq 1$ in \eqref{eq:delay_loo_main} and say that you assume the condition holds for all $t,\nu$ such that $t \geq \nu \geq 1$.}
For any $n \geq \nu \geq 1$ and $\eps > 0$,
\begin{align}
\label{eq:delay_loo_main2}
    &\Prob{\nu}{\sum_{i=n}^{n+n_b-1} \widehat{Z}^{n+n_b-1,n}_i < b} \nonumber\\
    &= \Prob{\nu}{\sum_{i=n}^{n+n_b-1} \widehat{Z}^{n+n_b-1,n}_i < b, \sum_{i=n}^{n+n_b-1} Z_i - \widehat{Z}^{n+n_b-1,n}_i \leq \eps} + \nonumber\\
    &\qquad \Prob{\nu}{\sum_{i=n}^{n+n_b-1} \widehat{Z}^{n+n_b-1,t}_i < b, \sum_{i=n}^{n+n_b-1} Z_i - \widehat{Z}^{n+n_b-1,t}_i \geq \eps} \nonumber\\
    &\leq \Prob{\nu}{\sum_{i=n}^{n+n_b-1} Z_i \leq b + \eps} + \Prob{\nu}{\frac{1}{n_b} \sum_{i=n}^{n+n_b-1} \left(Z_i - \widehat{Z}^{n+n_b-1,t}_i\right) \geq \frac{\eps}{n_b}} \nonumber\\
    &= \Prob{1}{\sum_{i=1}^{n_b} Z_i \leq b + \eps} + \Prob{1}{\frac{1}{n_b} \sum_{i=1}^{n_b} \left(Z_i - \widehat{Z}^{n_b,1}_i\right) \geq \frac{\eps}{n_b}}
\end{align}
where $Z_i$ is the true log-likelihood ratio at time $i$. Observe that the first term increases with $\eps$, while the second term decreases. It is important to choose a proper $\eps = \eps_b$ in order to keep both terms small. The idea in the following is that we first choose a proper $\eps = \eps_b$ by controlling the second term, and then verify that it is small enough for the first term when $b$ becomes large.

Below, the goal is to choose $\eps = \eps_b$ and $\delta_b$ such that
\[ \Prob{1}{\sum_{i=1}^{n_b} Z_i < b + \eps_b} \leq \delta_b^2 \quad \text{and}\quad \Prob{1}{\frac{1}{n_b} \sum_{i=1}^{n_b} \left(Z_i - \widehat{Z}^{n_b,1}_i\right) \geq \frac{\eps_b}{n_b}} \leq \delta_b^2 \]
hold simultaneously. In the following, write $\widehat{p}_i$ and $\widehat{Z}_i$ as short-hand notations for $\widehat{p}^{n_b,1}_{-i}$ and $\widehat{Z}^{n_b,1}_{-i}$, respectively.
% \vvv{This is confusing since \eqref{eq:delay_loo_main2} does not have $\widehat{p}^{n_b,1}_{-i}$. Maybe just replace ``For the second term in \eqref{eq:delay_loo_main2}," by ``In the following,"}
Note that $\E{1}{Z_i - \widehat{Z}_i} = \E{1}{\KL{p_1}{\widehat{p}_i}}$. Under the conditions for the estimator in \eqref{eq:converg_m1_bound_loo} and \eqref{eq:converg_m2_bound_loo}, the mean and variance of $n_b^{-1} \sum_{i=1}^{n_b} (Z_i - \widehat{Z}_i)$ can be bounded as
\begin{align} \label{eq:delay_loo_mean_var_bound}
    &\E{1}{\frac{1}{n_b} \sum_{i=1}^{n_b} (Z_i - \widehat{Z}_i)} = \E{1}{\frac{1}{n_b} \sum_{i=1}^{n_b} \log\frac{p_1(X_i)}{\widehat{p}_i(X_i)}} \leq \frac{C_1}{n_b^{\beta_1}} \nonumber \\
    &\Var{1}{\frac{1}{n_b} \sum_{i=1}^{n_b} (Z_i - \widehat{Z}_i)} = \Var{1}{\frac{1}{n_b} \sum_{i=1}^{n_b} \log\frac{p_1(X_i)}{\widehat{p}_i(X_i)}} \leq \frac{C_2}{n_b^{\beta_2}}.
\end{align}
% \vvv{I would replace $\E{1}{\log\frac{p_1(X_i)}{\widehat{p}_i(X_i)}}$ by $\E{1}{(Z_i - \widehat{Z}_i)}$ in the above equation. I know they are the same thing but it would cause less confusion.}

Now,
\begin{align}
\label{eq:delay_loo_term2}
    &\Prob{1}{\frac{1}{n_b} \sum_{i=1}^{n_b} \left(Z_i - \widehat{Z}_i\right) \geq \frac{\eps_b}{n_b}} \nonumber\\
    &\leq \Prob{1}{\abs{\frac{1}{n_b} \sum_{i=1}^{n_b} \left(Z_i - \widehat{Z}_i\right) - \E{1}{\KL{p_1}{\widehat{p}_i}}} \geq \frac{\eps_b}{n_b} - \E{1}{\KL{p_1}{\widehat{p}_i}}} \nonumber\\
    &\stackrel{(*)}{\leq} \Var{1}{\frac{1}{n_b} \sum_{i=1}^{n_b} \log\frac{p_1(X_i)}{\widehat{p}_i(X_i)}} \left(\frac{\eps_b}{n_b} - \E{1}{\KL{p_1}{\widehat{p}_i}}\right)^{-2} \nonumber\\
    &\leq \frac{C_2}{n_b^{\beta_2}} \left(\frac{\eps_b}{n_b} - \E{1}{\KL{p_1}{\widehat{p}_i}}\right)^{-2}
\end{align}
for any $\eps_b > n_b \times \E{1}{\KL{p_1}{\widehat{p}_i}}$. Here $(*)$ follows from Chebyshev's inequality.
Now, \eqref{eq:delay_loo_term2} is less than or equal to $\delta_b^2$ if
\[
 \frac{\eps_b}{n_b} - \E{1}{\KL{p_1}{\widehat{p}_i}} \geq \frac{\sqrt{C_2}}{\delta_b n_b^{\frac{\beta_2}{2}}}
\]
which is equivalent to
\begin{equation}
\label{eq:delay_loo_eps_eqn}
    \eps_b \geq \frac{\sqrt{C_2} n_b^{1-\beta_2 / 2}}{\delta_b} + n_b \E{1}{\KL{p_1}{\widehat{p}_i}}.
\end{equation}
Consider the two terms on the right-hand-side of \eqref{eq:delay_loo_eps_eqn}. Since $\E{1}{\KL{p_1}{\widehat{p}_i}} \leq C_1 n_b^{-\beta_1}$ (from \eqref{eq:delay_loo_mean_var_bound}), the second term in \eqref{eq:delay_loo_eps_eqn} is no larger than $C_1 n_b^{1-\beta_1}$.
In order to choose a proper $\eps_b$, there are three cases depending on the rate of the first term 
%(as compared to that of the second term, which is $O(n_b^{1-\beta_1})$) 
in \eqref{eq:delay_loo_eps_eqn}.
\begin{itemize}
\item Case 1: $4 \beta_1 > \beta_2$. Let
\begin{equation}
\label{eq:delay_loo_choose_eps}
    \eps_b = \frac{2 \sqrt{C_2} n_b^{1- \beta_2 / 2}}{\delta_b},
\end{equation}
with $\delta_b$ as chosen below.
% And $\delta_b$ will be chosen to ensure that
% \[ \sqrt{C_2} \delta_b^{-1} n_b^{1-\beta_2 / 2} = \omega(n_b^{1-\beta_1})\]
% so that \eqref{eq:delay_loo_eps_eqn} is satisfied for large $b$.
With this $\eps_b$, the first term in \eqref{eq:delay_loo_main2} becomes
\begin{align*}
    & \Prob{1}{\sum_{i=1}^{n_b} Z_i < b + \eps_b} \nonumber\\
    &= \Prob{1}{\sum_{i=1}^{n_b} Z_i < (1-\delta_b) n_b I + \frac{2 \sqrt{C_2} n_b^{1-\beta_2 / 2}}{\delta_b}}\\
    &= \Prob{1}{\sum_{i=1}^{n_b} Z_i < (1-\delta_b) n_b I \brc{1 + \frac{2 \sqrt{C_2} }{(1-\delta_b)\delta_b n_b^{\beta_2 / 2} I} } } \nonumber\\
    &\leq \Prob{1}{\sum_{i=1}^{n_b} Z_i < (1-\delta_b) n_b I \brc{1 + \frac{2 \eta \sqrt{C_2}}{\delta_b n_b^{\beta_2 / 2} I } } }
\end{align*}
where in the last inequality we have used the fact that $1-\delta_b > \eta^{-1}$. Let
\begin{equation}
\label{eq:delay_loo_choose_delta}
    \delta_b = \frac{(4 \eta^2 C_2)^\frac{1}{4}}{n_b^{\beta_2 / 4} \sqrt{I}} \iff \frac{2 \eta \sqrt{C_2}}{\delta_b n_b^{\beta_2 / 2} I } = \delta_b.
\end{equation}
With this chosen $\delta_b$,
\begin{equation*}
    \Prob{1}{\sum_{i=1}^{n_b} Z_i < b + \eps_b} \leq \Prob{1}{\sum_{i=1}^{n_b} Z_i < (1-\delta_b^2) n_b I}.
\end{equation*}
Assuming that \eqref{eq:lai_lower} is true for $Z_i$'s, we have \cite[Appendix~B]{lai1998infobd}
\begin{equation*} 
    \Prob{1}{\sum_{i=1}^{n_b} Z_i < (1-\delta_b^2) n_b I} \leq \delta_b^2,
\end{equation*}
and thus
\begin{equation} 
\label{eq:delay_loo_true_llr}
    \Prob{1}{\sum_{i=1}^{n_b} Z_i < b + \eps_b} \leq \delta_b^2.
\end{equation}

Now, we verify that \eqref{eq:delay_loo_eps_eqn} holds for all large enough $b$'s. With the chosen $\delta_b$ (in \eqref{eq:delay_loo_choose_delta}), the first term in \eqref{eq:delay_loo_eps_eqn} satisfies
\[ \sqrt{C_2} \delta_b^{-1} n_b^{1-\beta_2 / 2} = \Theta(n_b^{1- \beta_2/4}) = \omega(n_b^{1-\beta_1}). \]
Therefore, the chosen $\eps_b$ (in \eqref{eq:delay_loo_choose_eps}) satisfies \eqref{eq:delay_loo_eps_eqn} for all $b$'s large enough. As a result, from \eqref{eq:delay_loo_term2}, we get
\[ \Prob{1}{\frac{1}{n_b} \sum_{i=1}^{n_b} \left(Z_i - \widehat{Z}_i\right) \geq \frac{\eps_b}{n_b}} \leq \delta_b^2. \]

% which, by \eqref{eq:delay_loo_choose_eps}, also implies that $\eps_b = \Theta\brc{ n_b^{1 - \beta_2 / 4}}$.
% As $b \to \infty$, $n_b \to \infty$ as well. Since $0 < \beta_2 < 2$, $\delta_b \searrow 0$, $\eps_b \nearrow \infty$, and $\eps_b = o(n_b)$. Now, 

\item Case 2: $4 \beta_1 < \beta_2$. Let
\begin{align}
\label{eq:delay_loo_choose_eps_delta_2}
    &\eps_b = 2 C_1 n_b^{1-\beta_1}, \nonumber\\
    &\delta_b = \frac{2 \eta C_1}{I} n_b^{-\beta_1} \iff \frac{\eta \eps_b}{n_b I} = \delta_b.
\end{align}
With this choice,
\begin{align} \label{eq:delay_loo_first_term_bound_2}
    \Prob{1}{\sum_{i=1}^{n_b} Z_i < b + \eps_b} &= \Prob{1}{\sum_{i=1}^{n_b} Z_i < (1-\delta_b) n_b I + 2 C_1 n_b^{1-\beta_1}} \nonumber\\
    &= \Prob{1}{\sum_{i=1}^{n_b} Z_i < (1-\delta_b) n_b I \brc{1 + \frac{2 C_1 }{(1-\delta_b) n_b^{\beta_1} I} } } \nonumber\\
    &\leq \Prob{1}{\sum_{i=1}^{n_b} Z_i < (1-\delta_b) n_b I \brc{1 + \frac{2 \eta C_1}{n_b^{\beta_1} I } } } \nonumber\\
    &= \Prob{1}{\sum_{i=1}^{n_b} Z_i < (1-\delta_b^2) n_b I}
\end{align}
and thus, assuming \eqref{eq:lai_lower} holds, we have
\[ \Prob{1}{\sum_{i=1}^{n_b} Z_i < b + \eps_b} \leq \delta_b^2. \]
Also, since
\[ \sqrt{C_2} \delta_b^{-1} n_b^{1-\beta_2 / 2} = \Theta\brc{n_b^{1-\beta_2/2+\beta_1}} = o(n_b^{1-\beta_1}), \]
the chosen $\eps_b$ (in \eqref{eq:delay_loo_choose_eps_delta_2}) satisfies \eqref{eq:delay_loo_eps_eqn} for all $b$'s large enough. As a result, from \eqref{eq:delay_loo_term2}, we get
\[ \Prob{1}{\frac{1}{n_b} \sum_{i=1}^{n_b} \left(Z_i - \widehat{Z}_i\right) \geq \frac{\eps_b}{n_b}} \leq \delta_b^2. \]

% The rest of the proof is similar, and in the same way we get \eqref{eq:delay_loo_true_llr}.
% With this choice of $\eps_b$ and $\delta_b$, we can get \eqref{eq:delay_loo_true_llr} with a similar argument.
% We want to verify that the first term in \eqref{eq:delay_loo_eps_eqn} is dominated by $n_b^{1-\beta_1}$ with $\delta_b$ defined in \eqref{eq:delay_loo_choose_eps_delta_2}.
% For verification, $\sqrt{C_2} \delta_b^{-1} n_b^{1-\beta_2 / 2} = \Theta(n_b^{1 + \beta_1 - \beta_2/2}) = o(n_b^{1-\beta_1})$ if $4 \beta_1 < \beta_2$. As a result, \eqref{eq:delay_loo_eps_eqn} is satisfied for all large enough $b$.

\item Case 3: $4 \beta_1 = \beta_2$. Let $C_3$ be a large enough constant such that
\begin{equation} \label{eq:delay_loo_choose_c3_3}
    C_3 \geq \frac{I \sqrt{C_2}}{\eta C_3} + C_1.
\end{equation}
Choose
\begin{align}
\label{eq:delay_loo_choose_eps_delta_3}
    &\eps_b = C_3 n_b^{1-\beta_1}, \nonumber\\
    &\delta_b = \frac{\eta C_3}{I} n_b^{-\beta_1} \iff \frac{\eta \eps_b}{n_b I} = \delta_b.
\end{align}
Following the same line of argument as in \eqref{eq:delay_loo_first_term_bound_2}, we get
\[ \Prob{1}{\sum_{i=1}^{n_b} Z_i < b + \eps_b} = \Prob{1}{\sum_{i=1}^{n_b} Z_i < (1-\delta_b) n_b I + 2 C_1 n_b^{1-\beta_1}} \leq \delta_b^2. \]
Also, from \eqref{eq:delay_loo_choose_c3_3},
\[ \eps_b = C_3 n_b^{1-\beta_1} \geq \frac{I \sqrt{C_2}}{\eta C_3} n_b^{1-\beta_2/2+\beta_1} + C_1 n_b^{1-\beta_1} = \frac{\sqrt{C_2} n_b^{1-\beta_2 / 2}}{\delta_b} + C_1 n_b^{1-\beta_1}. \]
Therefore, \eqref{eq:delay_loo_eps_eqn} is satisfied for the chosen $\eps_b$ (in \eqref{eq:delay_loo_choose_eps_delta_3}), and from \eqref{eq:delay_loo_term2} we get
\[ \Prob{1}{\frac{1}{n_b} \sum_{i=1}^{n_b} \left(Z_i - \widehat{Z}_i\right) \geq \frac{\eps_b}{n_b}} \leq \delta_b^2. \]
\end{itemize}

To sum up, in all cases, we have shown the existence of $\eps_b$ and $\delta_b$ (that depend on $\beta_1$ and $\beta_2$) such that
\[ \Prob{1}{\sum_{i=1}^{n_b} Z_i < b + \eps_b} \leq \delta_b^2 \quad \text{and}\quad \Prob{1}{\frac{1}{n_b} \sum_{i=1}^{n_b} \left(Z_i - \widehat{Z}^{n_b,1}_i\right) \geq \frac{\eps_b}{n_b}} \leq \delta_b^2 \]
hold simultaneously. Continuing \eqref{eq:delay_loo_main2}, we can write, for any $(n,\nu)$ such that $n \geq \nu \geq 1$,
% \vvv{Refer the reader back to the equation that talks about what you are trying to prove here rather than afterwards}
\begin{align*}
    &\Prob{\nu}{\sum_{i=n}^{n+n_b-1} \widehat{Z}^{n+n_b-1,n}_i < b}\\
    &\leq \Prob{1}{\sum_{i=1}^{n_b} Z_i \leq b + \eps_b} + \Prob{1}{\frac{1}{n_b} \sum_{i=1}^{n_b} \left(Z_i - \widehat{Z}^{n_b,1}_i\right) \geq \frac{\eps_b}{n_b}}\\
    &\leq 2 \delta_b^2.
\end{align*}
This is exactly what was required to be shown in \eqref{eq:delay_loo_main}. The proof is now complete. \qedhere
\end{proof}

\begin{proof}[Proof of Lemma~\ref{lem:fa_wla}]
Define the SR-like statistic
\[ R_n = (1 + R_{n-1}) e^{\widehat{Z}^w_n},~\forall n > w \]
with $R_1 = \dots = R_w = 0$. Also define the corresponding test:
\[ \overline{\tau}_R(b) := \inf \left\{n > w: R_n \geq e^b \right\}. \]
Note that the NWLA-CuSum statistic in \eqref{wla:stat} can be written equivalently as
\[ e^{\overline{W}(n)} = \max\cbrc{1, e^{\overline{W}(n-1)}} e^{\widehat{Z}^w_n},\quad n > w. \]
Therefore, for $n > w$, $R_n > e^{\overline{W}(n)}$ and thus $\overline{\tau}(b) \geq \overline{\tau}_R(b)$ on $\{\overline{\tau}(b) < \infty\}$.

Now, without loss of generality assume $\E{\infty}{\overline{\tau}(b)} < \infty$; otherwise the statement of the lemma holds  trivially. This implies that $\E{\infty}{\overline{\tau}_R(b)} < \infty$. Observe that $R_n \in {\cal F}_n$ and
\[ \E{\infty}{R_n - n| {\cal F}_{n-1}} = (1+R_{n-1}) \E{\infty}{e^{\widehat{Z}^w_n} | {\cal F}_{n-1}} -n = R_{n-1} - (n-1),~\forall n > w. \]
The last equality follows because $\widehat{p}^w_n$ is a density given ${\cal F}_{n-1}$. Hence $\{R_n - n\}_{n > w}$ is a $(\mathbb{P}_\infty,{\cal F}_n)$-martingale.
Also, for any $n > w$, since $R_n \in (0,e^b)$ almost surely on the event $\{\overline{\tau}_R(b) > n\}$, we have, for any $n > w$,
\begin{align*}
    &\E{\infty}{\abs{(R_{n+1} - (n+1)) - (R_n - n)}\Big|{\cal F}_{n}}\\
    &= \E{\infty}{\abs{R_{n+1} - R_n - 1}\Big|{\cal F}_{n}}\\
    &\leq \E{\infty}{R_{n+1}|{\cal F}_{n}} + (R_n + 1)\\
    &= 2 (R_n+1)\\
    &\leq 2 (e^b + 1)
\end{align*}
almost surely on the event $\{\overline{\tau}_R(b) > n\}$. Therefore, we can apply the optional sampling theorem and obtain
\[ \E{\infty}{R_{\overline{\tau}_R(b)}-\overline{\tau}_R(b)} = \E{\infty}{R_{w+1}-(w+1)} = \E{\infty}{e^{\widehat{Z}^w_{w+1}}} - (w+1) = -w, \]
where $\E{\infty}{e^{\widehat{Z}^w_{w+1}}} = 1$ because $\widehat{p}^w_{w+1}$ is a density given ${\cal F}_{w}$. Finally, we arrive at
\[ \E{\infty}{\overline{\tau}(b)} \geq \E{\infty}{\overline{\tau}_R(b)} = w + \E{\infty}{R_{\overline{\tau}_R(b)}} \geq e^b. \qedhere \]
\end{proof}

\begin{proof}[Proof of Lemma~\ref{lem:wla_esssup}]
The proof is similar to \cite[Lemma~4]{xie2022windowlimited}. Define a helping stopping time
\[ \overline{\tau}_\nu (b) := \inf\{n \geq \nu+w: \overline{W}_\nu (n) \geq b \} \]
where
\[ \overline{W}_\nu (n) = \brc{\overline{W}_\nu (n-1)}^+ + \widehat{Z}^w_n,\quad n \geq \nu + w \]
with $\overline{W}_\nu (n) = 0,\forall n < \nu + w$. Note that $\overline{W}(\nu+w) \geq \overline{W}_\nu(\nu+w)$. Now, if $\overline{W}(k) \geq \overline{W}_\nu(k)$, we have
\[ \overline{W}(k+1) = \brc{\overline{W}(k)}^+ + \widehat{Z}^w_{k+1} \geq \brc{\overline{W}_\nu(k)}^+ + \widehat{Z}^w_{k+1} = \overline{W}_\nu(k+1) \]
as long as $\overline{W}(k) < b$. Thus, by induction,
\[ \overline{W}(n) \geq \overline{W}_\nu(n),~\forall n \geq \nu+w\]
on the event $\cbrc{\overline{W}(n) < b}$, which implies that  $\overline{\tau}(b) \leq \overline{\tau}_\nu(b)$ almost surely under $\mathbb{P}_\nu$. In the remainder of the proof, we omit ``$(b)$" in the descriptions of the stopping times for notational brevity.

Since $\overline{\tau} - \nu + 1 = w + \brc{\overline{\tau} - \nu - w + 1} \leq w + (\overline{\tau} - \nu - w + 1)^+$, we have
\[ (\overline{\tau} - \nu + 1)^+ \leq w + (\overline{\tau} - \nu - w + 1)^+ \quad \mathbb{P}_\nu\text{--}a.s. \]
Thus,
% \vvv{[Since $\overline{\tau}_\nu \geq \nu + w - 1$, we can drop the positive part from last line below and cancel the $w$ from inside and outside the expectation.]} \yuchen{I am not sure I understand this comment. There is non-zero probability under $\mathbb{P}_\nu$ such that $\overline{\tau} < \nu + w - 1$ (e.g., in a false alarm event).}
\begin{align*}
    &\E{\nu}{(\overline{\tau} - \nu + 1)^+ | {\cal F}_{\nu-1} }\\
    &\leq w + \E{\nu}{(\overline{\tau} - \nu - w + 1)^+ | {\cal F}_{\nu-1} }\\
    % &= w + \E{\nu}{(\overline{\tau} - \nu - w + 1) \ind{\overline{\tau} - \nu - w + 1 \geq 0} | {\cal F}_{\nu-1} }\\
    &= w + \E{\nu}{\E{\nu}{(\overline{\tau} - \nu - w + 1)^+ | X_{\nu},\dots,X_{\nu+w-1}, {\cal F}_{\nu-1} }| {\cal F}_{\nu-1} }\\
    &\stackrel{(*)}{\leq} w + \E{\nu}{\E{\nu}{(\overline{\tau}_\nu - \nu - w + 1)^+ | X_{\nu},\dots,X_{\nu+w-1}, {\cal F}_{\nu-1} }| {\cal F}_{\nu-1} }\\
    &\stackrel{(**)}{=} w + \E{\nu}{\E{\nu}{\overline{\tau}_\nu - \nu - w + 1 | X_{\nu},\dots,X_{\nu+w-1}, {\cal F}_{\nu-1} }| {\cal F}_{\nu-1} }
\end{align*}
where $(*)$ holds because $\overline{\tau} \leq \overline{\tau}_\nu$ almost surely (under $\mathbb{P}_\nu$), and $(**)$ holds because $\overline{\tau}_\nu \geq \nu + w - 1 \geq 0$ almost surely (under $\mathbb{P}_\nu$).

Now, $\forall \nu \geq 1$, given the information of $X_{\nu},\dots,X_{\nu+w-1}$, the event $\cbrc{\overline{\tau}_\nu \geq \nu + w}$ is independent of ${\cal F}_{\nu-1}$. Thus,
\begin{align*}
    & \esssup \E{\nu}{\E{\nu}{\overline{\tau}_\nu - \nu - w + 1 | X_{\nu},\dots,X_{\nu+w-1}, {\cal F}_{\nu-1} } | {\cal F}_{\nu-1} } \\
    &= \E{\nu}{\E{\nu}{\overline{\tau}_\nu - \nu - w + 1 | X_{\nu},\dots,X_{\nu+w-1} } }\\
    &= \E{1}{\E{1}{\overline{\tau}_1 - w | X_{1},\dots,X_{w} }}\\
    % &= \E{1}{(\overline{\tau} - w)^+ }\\
    &= \E{1}{\overline{\tau} - w}
\end{align*}
% where the last line holds because $\overline{\tau} \geq w$ almost surely under $\mathbb{P}_1$.
The last line holds because $\overline{\tau}_1 = \overline{\tau}$ almost surely (under $\mathbb{P}_1$). The proof is now complete.
\end{proof}

\begin{proof}[Proof of Lemma~\ref{lem:delay_wla}]
The proof consists of two parts. In the first part, we use a similar technique as in \cite[Thm~1.1]{janson1983renewaldependent} to obtain an extension of Wald's identity to the case where the samples are $w$-dependent. In the second part, we upper bound the overshoot using results from renewal theory. For notational brevity, we omit the dependence on $w$ and write $\widehat{Z}^w_i = \widehat{Z}_i$.

Define
\[ Y_i := \E{1}{U_{i} - (i-w) \widehat{I}|{\cal F}_{i-w}},~\forall i \geq w\]
and note that $Y_w=0$. Now,
\begin{align*}
    \E{1}{Y_{i+1} | {\cal F}_{i-w}} &= \E{1}{\E{1}{U_{i+1} - (i+1-w) \widehat{I}|{\cal F}_{i+1-w}} | {\cal F}_{i-w}}\\
    &= \E{1}{U_{i} + \widehat{Z}_{i+1} - (i-w) \widehat{I} - \widehat{I} | {\cal F}_{i-w}}\\
    &= Y_i + \E{1}{\widehat{Z}_{i+1} - \widehat{I} | {\cal F}_{i-w}}\\
    &\stackrel{(*)}{=} Y_i + \E{1}{\widehat{Z}_{i+1} - \widehat{I}}\\
    &= Y_i,~\forall i > w,~\mathbb{P}_1\text{--}a.s.
\end{align*}
where $(*)$ follows from independence between $\widehat{Z}^w_{i+1}$ and ${\cal F}_{i-w}$. This implies that $\cbrc{(Y_i, {\cal F}_{i-w})}_{i \geq w}$ is a martingale. Therefore, for any finite $k > w$, $\min\cbrc{\tau_u,k} \leq k < \infty$, and thus
\begin{align*}
    &\E{1}{U_{\min\cbrc{\tau_u,k}} - \widehat{I} (\min\cbrc{\tau_u,k}-w)} \\
    &= \sum_{m=1}^\infty \Prob{1}{\tau_u = m} \E{1}{\E{1}{U_{\min\cbrc{m,k}} - \widehat{I} (\min\cbrc{m,k}-w)|{\cal F}_{\min\cbrc{m,k}-w}} \big| \tau_u = m}\\
    &= \sum_{m=1}^\infty \Prob{1}{\tau_u = m} \E{1}{Y_{\min\cbrc{m,k}} | \tau_u = m}\\
    &= \E{1}{Y_{\min\cbrc{\tau_u,k}}} \stackrel{(*)}{=} \E{1}{Y_w} = 0
\end{align*}
where $(*)$ follows from optional sampling theorem. This implies that
\begin{equation} \label{eq:wla_delay_p1_middle_step}
    \E{1}{U_{\min\cbrc{\tau_u,k}}} = \widehat{I} (\E{1}{\min\cbrc{\tau_u,k}}-w).
\end{equation}
Note that $\tau_u < \infty$ with probability 1 under $\mathbb P_1$ by Lemma~\ref{lem:wla_stop}.
For $i > w$, let $\widehat{Z}_i^+ := \max\{0,\widehat{Z}_i\}$ and $\widehat{Z}_i^- := -\min\{0,\widehat{Z}_i\}$. Note that $\widehat{Z}_i^+,\widehat{Z}_i^- \geq 0$, $\widehat{Z}_i = \widehat{Z}_i^+ - \widehat{Z}_i^-$, and $U_n = \sum_{i=w+1}^n \brc{\widehat{Z}_i^+ - \widehat{Z}_i^-},\forall n > w$. Thus, we have
\begin{align*}
    \lim_{k \to \infty} \E{1}{U_{\min\cbrc{\tau_u,k}}} &= \lim_{k \to \infty} \E{1}{\sum_{i=w+1}^{\min\cbrc{\tau_u,k}} \widehat{Z}_i^+} - \lim_{k \to \infty} \E{1}{\sum_{i=w+1}^{\min\cbrc{\tau_u,k}} \widehat{Z}_i^-}\\
    &\stackrel{(i)}{=} \E{1}{\lim_{k \to \infty} \sum_{i=w+1}^{\min\cbrc{\tau_u,k}} \widehat{Z}_i^+} - \E{1}{\lim_{k \to \infty} \sum_{i=w+1}^{\min\cbrc{\tau_u,k}} \widehat{Z}_i^-}\\
    &\stackrel{(ii)}{=} \E{1}{\sum_{i=w+1}^{\tau_u} \widehat{Z}_i^+} - \E{1}{\sum_{i=w+1}^{\tau_u} \widehat{Z}_i^-}\\
    &= \E{1}{U_{\tau_u}}
\end{align*}
where $(i)$ follows from the monotone convergence theorem, and $(ii)$ is due to the fact that $\tau_u < \infty$ with probability 1. Also by the monotone convergence theorem,
\[ \lim_{k \to \infty} \E{1}{\min\cbrc{\tau_u,k}} = \E{1}{ \lim_{k \to \infty} \min\cbrc{\tau_u,k}} = \E{1}{\tau_u}. \]
Thus, taking the limit of $k$ on both sides of \eqref{eq:wla_delay_p1_middle_step},
\begin{equation}
\label{eq:extended_wald}
    \E{1}{U_{\tau_u}} = \lim_{k \to \infty} \E{1}{U_{\min\cbrc{\tau_u,k}}} = \lim_{k \to \infty} \widehat{I} (\E{1}{\min\cbrc{\tau_u,k}}-w) = \widehat{I} (\E{1}{\tau_u}-w).
\end{equation}

Now, denote
\[ L_i := \log \frac{\widehat{p}_i^w(X_i)}{p_1(X_i)},\quad \forall i > w. \]
By definition we have $\E{1}{\widehat{Z}_i - L_i} = I,\forall i > w$. The proof of \eqref{eq:extended_wald} is also applicable to $L^2_{\tau_u}$, which gives us
\begin{equation}
\label{eq:extended_wald_l2}
    \E{1}{\sum_{i=w+1}^{\tau_u} L^2_{i}} = \E{1}{L^2_{w+1}} (\E{1}{\tau_u}-w).
\end{equation}
Thus,
\begin{align}
\label{eq:delay_overshoot_bound}
    \E{1}{U_{\tau_u}-b} &= \E{1}{U_{\tau_u-1}-b+\widehat{Z}_{\tau_u}} < \E{1}{\widehat{Z}_{\tau_u}} = I + \E{1}{L_{\tau_u}} \nonumber\\
    &\stackrel{(i)}{\leq} I + \sqrt{\E{1}{L_{\tau_u}^2}} \leq I + \sqrt{\E{1}{\sum_{i=w+1}^{\tau_u} L_i^2}} \nonumber\\
    &\stackrel{(ii)}{=} I + \sqrt{\brc{\E{1}{\tau_u}-w} \E{1}{L_{w+1}^2}} \nonumber\\
    &\stackrel{(iii)}{\leq} I + \sqrt{\frac{C_2}{w^{\beta_2}} \brc{\E{1}{\tau_u}-w}}
\end{align}
for sufficiently large $w$. Here $(i)$ follows from Jensen's inequality, $(ii)$ follows from \eqref{eq:extended_wald_l2}, and $(iii)$ follows from assumption \eqref{eq:converg_m2_bound}. Denote $c_w := C_2 w^{-\beta_2}$ and $x := \E{1}{\tau_u}-w$. The goal below is to get an upper bound for $x$. Combining \eqref{eq:delay_overshoot_bound} with \eqref{eq:extended_wald}, we obtain
\[
     \sqrt{c_w x} + I \geq \widehat{I} x - b  
\]
which implies that
\begin{align*}
     x & \leq \frac{(2 (b+I) \widehat{I} + c_w) + \sqrt{(2 (b+I) \widehat{I} + c_w)^2 - 4 \widehat{I}^2 (b+I)^2}}{2 \widehat{I}^{2}}\\
     & \leq \frac{2 (b+I) \widehat{I} + c_w}{\widehat{I}^{2}}.
\end{align*}
Plugging this bound into \eqref{eq:delay_overshoot_bound} gives us
\begin{align*} \E{1}{U_{\tau_u}-b} & \leq I + \sqrt{\frac{2 (b+I) c_w \widehat{I} + c_w^2}{\widehat{I}^{2}} } \\
& \leq I + \frac{2 \sqrt{(b+I) c_w \widehat{I}} + \sqrt{2} c_w }{\widehat{I}} 
\end{align*}
where in the last inequality we use the fact that $\sqrt{u+v} \leq \sqrt{2 u} + \sqrt{2 v}$ for any $u,v > 0$. Therefore, combining with \eqref{eq:extended_wald}, we obtain
\begin{align*}
    \E{1}{\tau_u} &= w + \widehat{I}^{-1} \brc{b + \E{1}{U_{\tau_u}-b}}\\
    &\leq w + \widehat{I}^{-1} (b+I) + \brc{ 4 C_2 (b+I) w^{-\beta_2} \widehat{I}^{-3} }^{\frac{1}{2}} + \sqrt{2} \widehat{I}^{-2} C_2 w^{-\beta_2}.
\end{align*}
The proof is now complete since $\E{1}{\overline{\tau}(b)} \leq \E{1}{\tau_u(b)}$ by Lemma~\ref{lem:dominance}. \qedhere

\end{proof}

\end{document}